\documentclass[11pt,final]{amsart}
\usepackage{amsmath}
\usepackage{amsfonts}  
\usepackage{amssymb}
\usepackage{graphicx} 

\usepackage[color,notref,notcite]{showkeys}   
\definecolor{labelkey}{rgb}{0.6,0,1}

\usepackage[normalem]{ulem}
\newcounter{corr}
\definecolor{violet}{rgb}{0.580,0.,0.827}
\newcommand{\corr}[3]{\typeout{Warning : a correction remains in page
\thepage}
				\stepcounter{corr}        
				{\color{blue}\ifmmode\text{\,\sout{\ensuremath{#1}}\,}\else\sout{#1}\fi}
       {\color{red}#2}
       {\color{violet} #3}}
\numberwithin{equation}{section}
\usepackage{constants}
\newconstantfamily{ctrcst}{symbol=C,reset={bibunit}}
\def\ctel#1{\ensuremath{\Cl[ctrcst]{#1}}}
\def\cter#1{\ensuremath{\Cr{#1}}}
 
 \newtheorem{thm}{Theorem}[section]
 \newtheorem{lem}[thm]{Lemma}
 \newtheorem{exam}[thm]{Example}
 \newtheorem{cor}[thm]{Corollary}
 \newtheorem{rem}[thm]{Remark}
 \newtheorem{defn}[thm]{Definition}

 \def\Id{\mathop{\rm Id}\nolimits}

 \def\span{\mathop{\rm Span}\nolimits}

 \newcommand{\R}{{\mathbb R}}           
\newcommand{\N}{{\mathbb N}}        

\author{W. Arendt}
\address{Wolfgang Arendt, Institute of Applied Analysis, University of Ulm. Helmholtzstr. 18, D-89069 Ulm (Germany)} 
\email{wolfgang.arendt@uni-ulm.de}

\author{I. Chalendar}
\address{Isabelle Chalendar,  Universit\'e Gustave Eiffel, LAMA, (UMR 8050), UPEM, UPEC, CNRS, F-77454, Marne-la-Vallée (France)}
\email{isabelle.chalendar@univ-eiffel.fr}

\author{R. Eymard}
\address{Robert Eymard,  Universit\'e Gustave Eiffel, LAMA, (UMR 8050), UPEM, UPEC, CNRS, F-77454, Marne-la-Vallée (France)}
\email{robert.eymard@univ-eiffel.fr}
\title[Error estimate for linear parabolic problems ]{Space-time error estimates \\ for approximations of linear parabolic problems \\ with generalized time boundary conditions}
\keywords{ }
\subjclass[2010]{65N30,47A07,47A52,46B20}
\begin{document}

\begin{abstract}
We first give a general error estimate for the nonconforming approximation of a problem for which a Banach-Ne\v cas-Babu\v ska (BNB) inequality holds. This framework covers parabolic problems with general conditions in time (initial value problems as well as periodic problems) under minimal regularity assumptions. We consider approximations by two types of space-time discretizations, both based on a conforming Galerkin method in space. The first one is the Euler $\theta-$scheme. In this case, we show that the BNB inequality is always satisfied, and may require an extra condition on the time step for $\theta\le \frac 1 2$. The second one is the time discontinuous Galerkin method, where the BNB condition holds without any additional condition.
\end{abstract}

\maketitle

\section{Introduction}

In the case of an elliptic problem, posed under the form: find $u\in U$, where $U$ is a Hilbert space over $\mathbb{K}=\mathbb{R}$ or $\mathbb{C}$, such that
\[
 a(u,v) = L(v)\mbox{ for all }v\in U,
\]
where $a(\cdot,\cdot)$ is a continuous ($|a(u,v)|\le M \Vert u\Vert_U\Vert v\Vert_U$) coercive ($|a(u,u)|\ge \alpha \Vert u\Vert_U^2$) bilinear form on $U$ and $L$ is a continuous linear form, C\'ea's Lemma provides an optimal a priori error estimate for an approximation $u_h\in U_h$, where $U_h$ is a finite dimensional supspace of $U$, such that
\[
  a(u_h,v) = L(v)\mbox{ for all }v\in U_h.
\]
C\'ea's error estimate reads as follows:
\[
 \Vert u - u_h\Vert_U \le \frac {M}{\alpha} \min_{v\in U_h} \Vert u - v\Vert_U.
\]
Note that the same result holds if we relax the coercivity hypothesis to an inf-sup hypothesis, also called Banach-Ne\v cas-Babu\v ska condition \cite{XZ03,EG04,boiv2019app,ACE} and \cite[Theorem 9.42]{AU23}, that is
\[
\sup_{v\in U_h} |a(u,v)|\ge \alpha \Vert u\Vert_U \Vert v\Vert_U\mbox{ for all }u\in U_h.
\]
It is important to notice that these error estimates  hold under minimal regularity assumptions, since it is only assumed that $u\in U$ is the solution of the continuous elliptic problem.

\medskip

In the case of parabolic problems, one can find  many results on a priori error estimates  in the literature, assuming some additional regularity for the solution of the continuous problem (see for example \cite[Theorem 6.29]{EG04}).
One can also find a few results providing a priori error estimates similar to C\'ea's Lemma in the case of semi-discrete numerical schemes (continuous in time, discrete in space), see for example \cite{chryhou2002err,tant2016proj}. In the case of fully discrete time-space problems, let us cite 
 \cite{schwab2009parab,boiv2019app}, where optimal error estimates are derived for specific time-space approximations of parabolic problems. 
 In \cite{urban2012error}, a time Crank-Nicolson scheme is used, and optimal error estimates are provided. 
 In \cite[Theorem 6]{saito2021var}, an error estimate for the time discontinuous Galerkin scheme is given in terms of discrete norms for functions with minimal regularity in space and fixed regularity in time.

\medskip

Following the spirit of these papers, the aim of this article is to give an error estimate for linear parabolic problems, which holds
\begin{itemize}
\item for general  boundary conditions in time (including initial value problems as well as problems which are periodic or antiperiodic   in time),
\item for  solutions with minimal regularity, both in space and time,
 \item for a large class of schemes including the Euler $\theta-$schemes and the time discontinuous Galerkin scheme (which also includes the Euler implicit scheme).
\end{itemize}

This goal is achieved in this paper as follows.

\medskip

In Section \ref{sec:NCBNB} we provide a nonconforming framework for the approximation of a linear problem which can be written under a weak form against a suitable space of test functions, verifying a uniform Banach-Ne\v cas-Babu\v ska condition. Note that, in this framework, the spaces, where  the approximating solutions live, are not necessarily identical nor included in the space containing the solution of the continuous problem. We then obtain an optimal error estimate in the sense of C\'ea's Lemma. 
\medskip

In Section \ref{sec:pbcont} we show that parabolic problems with general time boundary conditions can be considered in the framework of  Section \ref{sec:NCBNB},  allowing  error estimates for nonconforming-in-time methods, which are easy to formulate.  They are valid even when  the solution of the continuous problem has only  minimal regularity.

\medskip

The proof that the explicit or implicit Euler scheme, or more generally the Euler $\theta-$scheme with $\theta\in[0,1]$, with conforming Galerkin approximation in space fit in our  framework is therefore done using the results of  Section \ref{sec:NCBNB} and Lemma \ref{lem:suffbnb} given in the appendix.   Here   stability conditions on the time step play an important role in the case $\theta\le \frac 1 2$. They are linked with the comparison   of norms in the finite dimensional setting. The implicit Euler scheme is also a particular case of the time discontinuous Galerkin method \cite{eriksson1985dg,saito2021var}, that we  investigate in Section \ref{sec:dg} using the same framework. Our approach is comparable to what is done in \cite{eym2018disc,gdm} for the space-discontinuous Galerkin method in the elliptic case, whereas we now treat the parabolic equation.

Appendix \ref{sec:hilbert} gives technical results on operators on Hilbert spaces needed for proving  uniform BNB estimates in the considered time discretizations. 
Appendix \ref{sec:gram} provides the results necessary for computing the coefficients of the inverse of the Gram matrix which are used  in  Section \ref{sec:dg}. \\  

In the whole paper $\N$ denotes the set of all non-negative integers, whereas $\N^\star$ is used for the set of all positive integers.  
All the vector spaces are considered over $\mathbb{R}$. 

\section{A non-conforming Banach-Ne\v cas-Babu\v ska theorem} \label{sec:NCBNB}

Before introducing our approximation theorem, we first study the relation between inf-sup and dual inf-sup conditions. It will be convenient to introduce the notion of limit of a sequence of subspaces of a Banach space (see Subsection~\ref{sec:2.2}).

\subsection{The inf-sup condition and its dual version}\label{sec:infsup}
Let $U$ and $V$ be normed vector spaces, and let $b:U\times V\to\mathbb{R}$ be a bilinear form. We say that the bilinear form $b$ satisfies the \textit{inf-sup condition} if  there exists $\beta>0$ such that
\[
 \sup_{v\in V,\Vert v\Vert_V \le 1}b(u,v)\ge\beta\Vert u\Vert_U\hbox { for all }u\in U.
\]
This is equivalent to \[\inf_{u\in U,\Vert u\Vert_U\leq 1} \sup_{v\in V,\Vert v\Vert_V \le 1} b(u,v)>0,\]
which justifies the terminology. 

If $b$ is continuous on $U\times V$, then, for all $u\in U$, $Bu = b(u,\cdot)$ defines a bounded linear mapping $B\in\mathcal{L}(U,V')$. If $U$ is complete, then the inf-sup condition is equivalent to $B$ being injective with closed range. So the inf-sup condition is an important step to prove invertibility of $B$. Indeed, if $V$ is a reflexive Banach space then $B$ is invertible if and only if $b$ satisfies the inf-sup condition and the \textit{dual uniqueness condition}
\begin{equation}\label{eq:2.3}
 \hbox{given }v\in V, \ b(u,v) = 0\hbox{ for all }u\in U\hbox{ implies } v=0.
\end{equation}
In fact,  if $V$ is reflexive,  then \eqref{eq:2.3} is equivalent to the density of  ${\rm ran}(B)$. We now study the relation between the inf-sup condition and the dual inf-sup condition.
\begin{lem}\label{lem:bnbstargen} Let $U$ and $V$ be normed vector spaces, $b:U\times V\to\mathbb{R}$  a bilinear form and $\beta>0$ such that
\begin{itemize}
 \item[(a)] $\displaystyle\sup_{v\in V,\Vert v\Vert_V \le 1}b(u,v)\ge\beta\Vert u\Vert_U$ for all $u\in U$,
 \item[(b)] $b(u,.)\in V'$ for all $u\in U$ such that   the mapping $B:U\to V'$, $u\mapsto b(u,\cdot)$ is surjective.
\end{itemize}
Then $\displaystyle\sup_{u\in U,\Vert u\Vert_U \le 1}b(u,v)\ge\beta\Vert v\Vert_V$ for all $v\in V$.
 \end{lem}
\begin{proof} Let $v_0\in V$ and let $\displaystyle\mathcal{N}(v_0) = \sup_{u\in U,\Vert u\Vert_U \le 1}b(u,v_0)$. Then
\[
 b(u,v_0)\le \Vert u\Vert_U \mathcal{N}(v_0)\hbox{ for all }u\in U.
\]
By the Hahn-Banach theorem, there exists $L\in V'$, with $\Vert L\Vert_{V'} =1$ and $\Vert v_0\Vert_V = L(v_0)$. By Hypothesis (b), there exists $u_0\in U$ such that
\[
 b(u_0,v) = L(v)\hbox{ for all }v\in V.
\]
From Hypothesis (a) we deduce that $\beta\Vert u_0\Vert_U\le \Vert L\Vert_{U'}$. This yields
\[
 \Vert v_0\Vert_V = L(v_0) = b(u_0,v_0)\le \Vert u_0\Vert_U \mathcal{N}(v_0) \le \frac 1 \beta \mathcal{N}(v_0),
\]
which proves the conclusion of the lemma.
\end{proof}
The following corollary will be  used later for obtaining the dual BNB-condition \eqref{eq:bnbstar} from \eqref{eq:betabnbn}.
\begin{cor}\label{cor:bnbstar} Let $U$ and $V$ be normed vector spaces having the same finite dimension, and let $b:U\times V\to\mathbb{R}$ be a bilinear form. Then the following assertions are equivalent for any $\beta>0$.
\begin{itemize}
 \item[(a)] $\displaystyle\sup_{v\in V,\Vert v\Vert_V \le 1}b(u,v)\ge\beta\Vert u\Vert_U$ for all $u\in U$,
 \item[(b)] $\displaystyle\sup_{u\in U,\Vert u\Vert_U \le 1}b(u,v)\ge\beta\Vert v\Vert_V$ for all $v\in V$.
\end{itemize}
 \end{cor}
 \begin{proof} For finite dimensional vector spaces with the same dimension, the surjectivity is a consequence of the one-to-one property. This, in turn, is an  immediate consequence of the inf-sup hypothesis.
 \end{proof}

\subsection{The limit of a sequence of subspaces for the strong and the weak topologies}\label{sec:2.2}

The following notion plays an important role in our paper.
\begin{defn}[Limit of a sequence of subspaces for the strong topology]\label{def:limseqspace}
 Let $Z$ be a Banach space and let $(X_n)_{n\in\mathbb{N}}$ be a sequence of subspaces of $Z$. We define the {\rm limit} of the sequence $(X_n)_{n\in\mathbb{N}}$ by
\begin{equation}\label{eq:defF}
\displaystyle\lim_{n\to\infty} X_n := \{x\in Z:~\hbox{for all }n\in\mathbb{N}\hbox{ there exists }x_n\in X_n\ \hbox{ such that }\displaystyle\lim_{n\to\infty} x_n = x\}.
\end{equation}
\end{defn}

\begin{lem}\label{lem:limFn}
Let $Z$ be a Banach space and let $(X_n)_{n\in\mathbb{N}}$ be a sequence of subspaces of $Z$. 
Then $X = \displaystyle\lim_{n\to\infty} X_n$ is closed. 
\end{lem}
\begin{proof}
 Let $z\in\overline{X}$. In order to construct $(x_n)_{n\in\mathbb{N}}$ such that $x_n\in X_n$ and $\displaystyle\lim_{n\to\infty} x_n = z$, we first construct by induction, for any $k\in\mathbb{N}$,
 a sequence  $(y_n^k)_{n\in\mathbb{N}}$ and a number $n_k$ such that
 \begin{itemize}
 \item $y_n^k\in X_n$ for all $n\in\mathbb{N}$.
  \item $n_k < n_{k+1}$,
  \item $\Vert z - y_n^k\Vert_Z \le \frac 1 {k+1}$ for any $n\ge n_k$.
 \end{itemize}
Indeed, let $k=0$. Let $x\in X$ be such that $\Vert z - x\Vert_Z \le \frac 1 2$. Let $(y_n^0)_{n\in\mathbb{N}}$ be such that $y_n^0\in X_n$  for all $n\in\mathbb{N}$ and $\displaystyle\lim_{n\to\infty} y_n^0 = x$. We select $n_0\in\mathbb{N}$ such that, for any $n\ge n_0$, $\Vert x - y_n^0\Vert_Z \le \frac 1 2$. Then $\Vert z - y_n^0\Vert_Z \le 1$ for any $n\ge n_0$.

Assume that, for a given $k\ge 1$, the construction for $k-1$ is done. Let $x\in X$ be such that $\Vert z - x\Vert_Z \le \frac 1 {2(k+1)}$. Let $(y_n^k)_{n\in\mathbb{N}}$ be such that $y_n^k\in X_n$ and $\displaystyle\lim_{n\to\infty} y_n^k = x$. We select $n_k\in\mathbb{N}$ such that $n_k>n_{k-1}$ and, for any $n\ge n_k$, $\Vert x - y_n^k\Vert_Z \le \frac 1 {2(k+1)}$. Then $\Vert z - y_n^k\Vert_Z \le \frac 1 {k+1}$ for any $n\ge n_k$.
  This finishes the proof of the inductive statement.
  
 We now set $x_n = y_n^0$ for $n= 0,\ldots,n_0 - 1$, and, for any $k\ge 1$,  $x_n = y_n^k$ for any $n = n_k,\ldots,n_{k+1}-1$.
 We then obtain that $x_n\in X_n$ for any $n\in\mathbb{N}$ and that
 \[
   \Vert z - x_n\Vert_Z \le \frac 1 {k+1}\mbox{ for all }n\ge n_k ,
 \]
 which shows that $\displaystyle\lim_{n\to\infty} x_n = z$, and therefore that $z\in X$.
\end{proof}

The following notion is also used in our paper.
\begin{defn}[Limit of a sequence of subspaces for the weak topology]\label{def:limseqspaceweak}
 Let $Z$ be a Banach space and let $(X_n)_{n\in\mathbb{N}}$ be a sequence of subspaces of $Z$. We define the {\rm weak limit} of the sequence $(X_n)_{n\in\mathbb{N}}$ by
\begin{equation}\label{eq:defopad}
 \mathop{\rm w\mbox{-}lim}\limits_{n\to\infty} X_n := \{x\in Z:\mbox{ there exist } x_{n}\in X_{n}\hbox{ such that } x_{n}\rightharpoonup x\hbox{ as }n\to\infty\}.
\end{equation}
where $\rightharpoonup$ stands for ``converges for the weak topology of $Z$''.  
\end{defn}
Note that we always have
\[
 \displaystyle\lim_{n\to\infty} X_n\subset \mathop{\rm w\mbox{-}lim}\limits_{n\to\infty} X_n.
\]
The determination of $\displaystyle\lim\limits_{n\to\infty} X_n$ and $\mathop{\rm w\mbox{-}lim}\limits_{n\to\infty} X_n$ is not always easy (see Corollary \ref{cor:4.20}). The following examples show various situations.
\begin{exam} Let $Z = \mathbb{R}^2$, and for all $n\in\mathbb{N}$, $X_{2n} = \{0\}\times \mathbb{R}$ and  $X_{2n+1} = \mathbb{R}\times \{0\}$. Then $\displaystyle\lim\limits_{n\to\infty} X_n = \mathop{\rm w\mbox{-}lim}\limits_{n\to\infty} X_n =\{(0,0)\}$ (in finite dimension, the weak and the strong topology are identical).
\end{exam}

\begin{exam}\label{exam:defX} Let $Z = L^2((0,1))$. Let $n\in\mathbb{N}^\star$, let $h= 1/n$, and let $X_n$ be the set of all functions which are constant on $((i-1)h,ih)$ for all $i=1,\ldots,n$. Then $\displaystyle\lim_{n\to\infty} X_n = \mathop{\rm w\mbox{-}lim}\limits_{n\to\infty} X_n = Z$.
\end{exam}
\begin{exam}\label{exam:defY}  $Z = L^2((0,1))$. Let $n\in\mathbb{N}^\star$, let $h= 1/n$, let $X_n$ be the set of all functions which are constant on $((i-\frac 1 2)h,ih)$ and null on $((i-1)h, (i-\frac 1 2)h)$ for all $i=1,\ldots,n$. Then $\displaystyle\lim_{n\to\infty} X_n = \{0\}$ and $ {\mathop{\rm w\mbox{-}lim}\limits_{n\to\infty} X_n} = Z$.  

Indeed, let $(y_n)_{n\in\mathbb{N}}$ with $y_n\in X_n$ be a sequence which converges in $L^2((0,1))$ to some function $y\in L^2((0,1))$. We extend the functions  $y_n$ and $y$ to $\mathbb{R}$ by 0. Then  $y(\cdot -\frac 1 {2n})$ converges to $y$ in $L^2$ as $n\to\infty$. We prove that $y_n(\cdot -\frac 1 {2n}) - y(\cdot -\frac 1 {2n})$ tends to $0$ in $L^2$,  using the change of variable $s\to s+\frac 1 {2n}$. We get that
$y_n(\cdot -\frac 1 {2n})$ converges to $y$ in $L^2$, and therefore $y_n y_n(\cdot -\frac 1 {2n})$ converges to $y^2$ in $L^1$. Observing that $y_n(s)y_n(s-\frac 1 {2n}) = 0$ for all $s\in\mathbb{R}$, we deduce that $y = 0$.

On the other hand, given  $y\in L^2((0,1))$, define $y_n\in X_n$ by the constant value $\frac 2 h\int_{(i-1)h}^{ih} y(s){\rm d}s$ on $((i-\frac 1 2)h,ih)$. Then $y_n\rightharpoonup y$ as $n\to\infty$, which shows that $ {\mathop{\rm w\mbox{-}lim}\limits_{n\to\infty} X_n} = Z$.
\end{exam}

\begin{lem}\label{exam:nonclosed} 
 On the space $Z = L^2((0,1))$, there exists a sequence $(X_n)_{n\in\mathbb{N}}$ of subspaces of $Z$ such that $\mathop{\rm w\mbox{-}lim}\limits_{n\to\infty} X_n$ is not closed.
\end{lem}
\begin{proof} Let $n\in\mathbb{N}^\star$, and let $X_n$ be the set of all functions which are null on $({\frac {i-1} n}, {\frac {i-u_{n,i}} n})$ and constant on $({\frac {i-u_{n,i}} n},{\frac i n})$, where $u_{n,i} = \frac 1 4 ({\frac i n})^4$, for all $i=1,\ldots,n$.
	Define, for a given $m\in\mathbb{N}$, the function $y^m:(0,1)\to\{0,1\}$ by the value $0$ on $(0,\frac 1 m)$ and $1$ on $(\frac 1 m,1)$. 
	
	We show that $y^m\in \mathop{\rm w\mbox{-}lim}\limits_{n\to\infty} X_n$.  Define $(y_n)_{n\in\mathbb{N}}$ by the constant value $0$ on $({\frac {i-u_{n,i}} n},{\frac i n})$ if $i<i_n$ and by the constant value $1/u_{n,i}$ on $({\frac {i-u_{n,i}} n},{\frac i n})$ if $i_n\le i\le n$ where ${\frac {i_n-1} n} < \frac 1 m\le {\frac {i_n} n}$. Then
\[
 \Vert y_n\Vert_{L^2}^2 = \sum_{i=i_n}^n\frac 1 {n u_{n,i}} \le 4 m^4.
\]
Owing to this bound in $L^2$, it is possible by density of $C^1_c((0,1))$ in $L^2((0,1))$ to prove $y_n\rightharpoonup y^m$ by showing that, for any $\varphi\in C^1_c((0,1))$, $\int_0^1\varphi(s)y_n(s){\rm d}s\to \int_{\frac 1 m}^1\varphi(s){\rm d}s$ as $n\to\infty$. Let $M_0 = \Vert \varphi\Vert_\infty$ and $M_1 = \Vert \varphi'\Vert_\infty$. We have
\[
 \Big\vert  n \int_{\frac {i-1} n}^{\frac {i} n}\varphi(s){\rm d}s - \frac n {u_{n,i}}\int_{\frac {i-u_{n,i}} n}^{\frac {i} n}\varphi(s){\rm d}s\Big\vert =\vert \varphi(s_1) - \varphi(s_2)\vert \le {M_1} |s_1 - s_2| \le \frac {M_1} n,
\]
where $s_1\in ({\frac {i-1} n}, {\frac {i} n})$ and $s_2\in ( {\frac {i-u_{n,i}} n},{\frac {i} n})$. Hence,
noticing that
\[
 \int_0^1\varphi(s)y_n(s){\rm d}s = \sum_{i=i_n}^n \frac 1 {u_{n,i}}\int_{\frac {i-u_{n,i}} n}^{\frac {i} n}\varphi(s){\rm d}s
\]
and
\[
  \int_{\frac 1 m}^1\varphi(s){\rm d}s =  -\int_{\frac {i_n-1} n}^{\frac 1 m}\varphi(s){\rm d}s+ \sum_{i=i_n}^n \int_{\frac {i-1} n}^{\frac {i} n}\varphi(s){\rm d}s,
\]
we can write
\[
 \Big\vert \int_0^1\varphi(s)y_n(s){\rm d}s - \int_{\frac 1 m}^1\varphi(s){\rm d}s\Big\vert \le \frac {M_0 + M_1} n.
\]

We deduce that $y_n\rightharpoonup y^m$ as $n\to\infty$, which shows that $y^m\in \mathop{\rm w\mbox{-}lim}\limits_{n\to\infty} X_n$.

We now observe that, as $m\to\infty$, $y^m$ converges in $Z$ to the constant function equal to $1$ on $(0,1)$, denoted by $y^\infty$. Let us prove that $y^\infty\notin \mathop{\rm w\mbox{-}lim}\limits_{n\to\infty} X_n$. Let us assume the contrary. Let $(y_n)_{n\in\mathbb{N}}$ with $y_n\in X_n$  such that $y_n\rightharpoonup y^\infty$ as $n\to\infty$. We denote by $y_{n,i}$ the constant value of $y_n$ on $({\frac {i-u_{n,i}} n},{\frac i n})$. Then 
\[
 M := \sup_{n\ge 1}\Vert y_n\Vert_{L^2} = \sup_{n\ge 1}\Big(\sum_{i=1}^{n}\frac {y_{n,i}^2 u_{n,i}} {n}\Big)^{1/2}<\infty .
\]
Let $0< a <\min(\frac 2 M,1)$. Define  $\varphi_a$ by $\varphi_a(s) = 1$ for $s\le a$ and $0$ for $s>a$. Then, choosing $i_n\in\{1,\ldots,n\}$ such that $\frac {i_n-1} n\le a< \frac {i_n} n$ and using the Cauchy-Schwarz inequality, we obtain 
\[
 \int_0^1 \varphi_a(s)y_n(s) {\rm d}s \le \sum_{i=1}^{i_n}\frac {|y_{n,i}| u_{n,i}} {n} \le \Big( \sum_{i=1}^{i_n}\frac {y_{n,i}^2 u_{n,i}} {n} \Big)^{1/2}\Big( \sum_{i=1}^{i_n}\frac {u_{n,i}} {n} \Big)^{1/2}.
\]
This gives
\[
  \int_0^1 \varphi_a(s)y_n(s) {\rm d}s \le M \Big(u_{n,i_n}\Big)^{1/2} = \frac M 2 (\frac {i_n} n)^2.
\]
Letting $n\to\infty$, we get
\[
 \int_0^1 \varphi_a(s)y^\infty(s) {\rm d}s = a \le \frac M 2 a^2.
\]
This contradicts  $a<\frac 2 M$.
\end{proof}

\subsection{A nonconforming approximation method}\label{sec:nonconfap}

Let $Z,Y$ be Banach spaces and let $b~:~Z\times Y\to \mathbb{R}$ be a continuous bilinear form, in the sense that there exists a constant $M>0$ such that
\begin{equation}\label{eq:bcontinuous}
|b(z,y)|\le M\Vert z\Vert_Z\Vert y\Vert_Y,\hbox{ for all }z\in Z,\, y\in Y. 
\end{equation}
Let $(X_n)_{n\in\mathbb{N}}$ and $(Y_n)_{n\in\mathbb{N}}$ be sequences of finite dimensional subspaces of $Z$ and $Y$, respectively, such that  $X_n$ and $Y_n$ have the same dimension for all  $n\in\mathbb{N}$. We assume that the  Banach-Ne\v cas-Babu\v ska condition (abbreviated by  \textit{BNB-condition}) is satisfied. This means that there exists $\beta>0$ such that
 \begin{equation}\label{eq:betabnbn}
 \sup_{y\in Y_n,\Vert y\Vert_Y = 1}b(x,y)\ge \beta\Vert x\Vert_Z,\hbox{ for all }x\in X_n\mbox{ for all }n\in\mathbb{N}.
 \end{equation}
  We deduce from \eqref{eq:betabnbn} an inf-sup condition on $\displaystyle\lim_{n\to\infty} X_n$.
 \begin{lem}\label{lem:infsup} Let $X = \displaystyle\lim_{n\to\infty} X_n$. 
Then
  \begin{equation}\label{eq:betabnb}
  \sup_{y\in Y,\Vert y\Vert_Y = 1}b(x,y)\ge \beta\Vert x\Vert_Z,\hbox{ for all }x\in X.
 \end{equation}
 \end{lem}
\begin{proof}
 Let $x\in X$ and let  $(x_n)_{n\in\mathbb{N}}$ be such that $x_n\in X_n$ and $\displaystyle\lim_{n\to\infty} x_n = x$ in $Z$. For any $n\in\mathbb{N}$, choose $y_n\in Y_n$ such that
 $\Vert y_n\Vert_Y = 1$ and $b(x_n,y_n)\ge \beta\Vert x_n\Vert_Z$. Then
 \[
  \sup_{y\in Y,\Vert y\Vert_Y = 1} b(x,y) \ge b(x,y_n) = b(x-x_n,y_n)+b(x_n,y_n)\ge -M\Vert x-x_n\Vert_Z + \beta\Vert x_n\Vert_Z.
 \]
Letting $n\to\infty$ in the preceding inequality yields \eqref{eq:betabnb}.
\end{proof}
Let $L\in Y'$. Then, owing to \eqref{eq:betabnbn}, there exists one and only one $x_n\in X_n$ such that
\begin{equation}\label{eq:pbinXn}
 x_n\in X_n, \ b(x_n,y) = L(y) \hbox{ for all }y\in Y_n.
\end{equation}
We now study the approximation of a solution $\widehat{x}$ of the following problem:
\begin{equation}\label{eq:pbinZ}
 \widehat{x}\in Z, \ b(\widehat{x},y) = L(y) \hbox{ for all }y\in Y.
\end{equation}
Theorem~\ref{thm:bnbnonconf} below provides sufficient conditions for the existence of such a solution.

Concerning approximations, we have the following result, which proves that, up to a constant, the discrete solutions converge with optimal speed.
\begin{lem}\label{lem:errest}
Let $L\in Y'$ and let  $\widehat{x}\in Z$ such that \eqref{eq:pbinZ} holds. For any $n\ge 0$, let $x_n$ be given by \eqref{eq:pbinXn}. Then we have
 \begin{equation}\label{eq:errestchixbanach}
\Vert x_n - \widehat{x}\Vert_Z\le (1+\frac {M} {\beta})\min_{x\in X_n}\Vert x - \widehat{x}\Vert_Z,
 \end{equation}
 \begin{equation}\label{eq:xxlul}
\Vert x_n - x\Vert_Z\le \frac{M}{\beta}\Vert x - \widehat{x}\Vert_Z\mbox{ for all }x\in X_n 
 \end{equation}
and, if $Z$ is a Hilbert space, 
 \begin{equation}\label{eq:errestchixhilbert}
\Vert x_n - \widehat{x}\Vert_Z\le \frac {M} {\beta}\min_{x\in X_n}\Vert x - \widehat{x}\Vert_Z.
 \end{equation}
\end{lem}
\begin{proof}
For any $n\in\mathbb{N}$, define $Q_n:Z\to X_n$   by 
\[
 b(Q_n z , y) = b(z,y), \mbox{ for all } y\in Y_n\mbox { and all }z\in Z.
\]
Then $Q_n\in \mathcal{L}(Z)$ is such that $Q_n^2 = Q_n$. Moreover, 
\[
  \beta \Vert Q_n z\Vert_Z \le  \sup_{y\in Y_n,\Vert y\Vert_Y = 1}b(z,y) \le M \Vert z\Vert_Z.
\]
This shows that
\begin{equation}\label{eq:normqn}
 \Vert Q_n\Vert \le \frac {M} {\beta}.
\end{equation}
Let now $L\in Y'$ be given and let $x_n$ be the solution to \eqref{eq:pbinXn}. We get that 
\[
 b(\widehat{x},y) = L(y) =b(x_n,y) \hbox{ for all }y\in Y_n.
\]
This proves that $Q_n \widehat{x} = x_n$. This first implies \eqref{eq:xxlul} owing to \eqref{eq:normqn}, since, for any $x\in X_n$,
\[
 x - x_n = Q_n(x-x_n) = Q_n(x-\widehat{x}).
\]
This also yields that
\begin{equation}\label{eq:galorth}
 \widehat{x} - x_n = \widehat{x} - x - Q_n(\widehat{x}  - x) = ({\rm Id} - Q_n)(\widehat{x}  - x),
\end{equation}
which proves \eqref{eq:errestchixbanach}  again by \eqref{eq:normqn}. 

Now we use an argument due to Xu and Zikatanov \cite{XZ03}, see also \cite[Theorem~9.42]{AU23}.
If $X_n \ne \{0\}$ and $X_n=Z$, by a result of Kato \cite[Lemma 4]{kato1060est}  (see also  \cite[Lemma 9.41]{AU23},
\[
 \Vert {\rm Id} - Q_n\Vert =\Vert Q_n\Vert \le \frac {M} {\beta}.
\]
Using \eqref{eq:galorth}, for any $x\in X_n$, we get
\[
 \Vert  \widehat{x} - x_n\Vert_Z\le \frac {M} {\beta} \Vert x - \widehat{x}\Vert_Z,
\]
which concludes the proof of \eqref{eq:errestchixhilbert}.  In the case $X_n = \{0\}$ or $X_n=Z$, the estimate \eqref{eq:errestchixhilbert} is trivial.
\end{proof}

Under an additional hypothesis on $X$ we now show that the solutions of the approximate problems converge to a solution of \eqref{eq:pbinZ}. Note that this solution $\widehat{x}$ becomes unique since we require $\widehat{x}\in X$ (and not merely $\widehat{x}\in Z$). Theorem~\ref{thm:bnbnonconf} will be used later to prove convergence of the Euler-time scheme
 (see Theorem~\ref{thm:cv}), and also convergence of the discontinuous Galerkin scheme in Section~\ref{sec:dg}.  It is a \textit{non-conforming convergence result} since the $X_n$ are not subspaces of $X$.

\begin{thm}\label{thm:bnbnonconf}
Let $Z,Y$ be Banach spaces with $Y$ reflexive and let $b:Z\times Y\to \mathbb{R}$ be a continuous bilinear form.  
Let $(X_n)_{n\in\mathbb{N}}$ and $(Y_n)_{n\in\mathbb{N}}$ be sequences of finite dimensional subspaces of $Z$ and $Y$, respectively, such that  $X_n$ and $Y_n$ have the same dimension for any $n\in\mathbb{N}$. We assume that there exists $\beta>0$ such that \eqref{eq:betabnbn} holds. Let $X_\infty =   \displaystyle\lim_{n\to\infty} X_n$. We assume that there exists a closed subspace $X$ of $X_\infty$ such that, given $y\in Y$, 
\begin{equation}\label{eq:denserange}
( b(x,y)=0\mbox{ for all }x\in X)\mbox{ implies } y=0.
\end{equation}
Then one has $X = X_\infty$. Moreover, for any $L\in Y'$, there exists one and only one solution $\widehat{x}\in X$ of \eqref{eq:pbinZ}. In addition, letting $x_n\in X_n$ be the solution of  \eqref{eq:pbinXn}, then $x_n$ converges to $\widehat{x}$ in $Z$.

Finally the identity $Y = {\mathop{\rm w\mbox{-}lim}\limits_{n\to\infty} Y_n}$ holds.
\end{thm}

\begin{proof}
 
Let $B:X \to Y'$ be defined by $ \langle Bx,y\rangle_{Y',Y} = b(x,y)$ for all $(x,y)\in X\times Y$. Since $b$ is continuous, $B$ is continuous as well. 
Since $X\subset X_\infty$,  \eqref{eq:betabnb} in Lemma \ref{lem:infsup} implies that  
\[
 \Vert B x\Vert_{Y'} \ge \beta \Vert x\Vert_Z\hbox{ for all }x\in X.
\]
This property means that $B$ is injective and has a closed range. Since $Y$ is reflexive,  \eqref{eq:denserange} implies that the range of $B$ is dense  in $Y'$. Hence $B$ is bijective, which proves the existence and uniqueness of $\widehat{x}\in X$ such that \eqref{eq:pbinZ}  holds. 

Let $x\in  X_\infty$. For $L\in Y'$ be defined by $L(y) = b(x,y)$ for all $y\in Y$, let  $\widehat{x}\in X$ be the solution of \eqref{eq:pbinZ}.  Then
\[
 b(\widehat{x}-x,y) = 0\mbox{ for all }y\in Y.
\]
Since $\widehat{x}-x\in X_\infty$, we deduce from \eqref{eq:betabnb} that $\widehat{x}-x = 0$. Hence $x = \widehat{x}\in X$, which proves that $X_\infty= X$.

\medskip

The convergence of $x_n$ to $\widehat{x}$ results from Lemma \ref{lem:errest}, since $\displaystyle\lim_{n\to\infty}X_n=X$. In fact,  Lemma~\ref{lem:errest} shows even that the convergence is at an optimal rate.  

Let us finally prove that $Y = {\mathop{\rm w\mbox{-}lim}\limits_{n\to\infty} Y_n}$ (see Definition \ref{def:limseqspaceweak}). 
From Corollary \ref{cor:bnbstar} we deduce that the dual BNB inequality holds, with the same $\beta$ as in \eqref{eq:betabnbn}. 
Let $y\in Y$ be given. For all $n\in\mathbb{N}$, let $y_n\in Y_n$ be defined by
\[
  b(x,y_n) = b(x,y)\mbox{ for all }x\in X_n.
\]
Using the preceding relation, we then get
 \begin{equation}\label{eq:bnbstar}
   \beta \|y_n\|_Y\leq \sup_{x\in X_n,\|x\|_Z=1}|b(x,y_n)|= \sup_{x\in X_n,\|x\|_Z=1}|b(x,y)|.
\end{equation}
This leads to
\[
 \beta \Vert y_n\Vert_Y \le M \Vert y\Vert_Y.
\]
Hence the sequence $(\Vert y_n\Vert_Y)_{n\in\mathbb{N}}$ is bounded. Since $Y$ is reflexive, there exists a subsequence $(y_{n_k})_{k\in\mathbb{N}}$ which converges to some $\widehat{y}\in Y$ for the weak topology of $Y$.
Let $x\in X$, and let $(x_n)_{n\in\mathbb{N}}$ with $x_n\in X_n$ converging to $x$ in $Z$. Then we have, for all $k\in\mathbb{N}$,
\[
 b(x_{n_k},y_{n_k}) =  b(x_{n_k},y).
\]
Passing to the limit $k\to\infty$ in the above relation implies $b(x,\widehat{y}) = b(x,y)$.
We then have
\[
 b(x,\widehat{y}-y) = 0\mbox{ for all }x\in X.
\]
By Hypothesis \eqref{eq:denserange}, this implies $\widehat{y}=y$. By uniqueness of the limit, the whole sequence $(y_n)_{n\in\mathbb{N}}$ converges to $y$ for the weak topology of $Y$, which concludes the proof.
\end{proof}
The proof of Theorem~\ref{thm:bnbnonconf} shows that also $X$ is necessarily reflexive. In fact, since $Y$ is supposed to be reflexive, also $Y'$ is reflexive and $B$ is an isomorphism between $X$ and $Y'$. However, in our applications in Section~\ref{sec:eulerex}   and \ref{sec:dg}, the space $Z$ will not be reflexive.    
\begin{exam} In general, one does not have $Y = \displaystyle\displaystyle\lim_{n\to\infty} Y_n$ under the hypotheses of Theorem \ref{thm:bnbnonconf}. Indeed, let $Z = Y = L^2((0,1))$, $b(z,y) = \int_0^1 z(s) y(s) {\rm d}s$. Let $(X_n)$ be defined by Example \ref{exam:defX} and $(Y_n)$ be defined by Example \ref{exam:defY}.  Then $X = \displaystyle\lim_{n\to\infty} X_n$, the BNB condition \eqref{eq:betabnbn} holds (one has $\beta = \sqrt{2}/2$: consider $x\in X_n$ defined by the values $x_i$, $i=1,\ldots,n$ on each interval $((i-1)h,ih)$, and $y\in Y_n$ defined by the values $x_i$, $i=1,\ldots,n$ on each interval $((i-\frac 1 2)h,ih)$) and the surjectivity condition \eqref{eq:denserange} holds as well. Then one has $\displaystyle\lim_{n\to\infty} Y_n = \{0\}$  whereas $Y = {\mathop{\rm w\mbox{-}lim}\limits_{n\to\infty} Y_n}$ as asserted by Theorem~\ref{thm:bnbnonconf}.  
\end{exam}

It is remarkable that, according to Lemma \ref{lem:errest}, the discrete solutions converge with optimal speed. This is what the estimates \eqref{eq:errestchixbanach} and \eqref{eq:errestchixhilbert} say, which are similar to the famous C\'ea's Lemma \cite{cea1964}, see also \cite[Theorem 9.14]{AU23}. In the conforming case $X_n\subset X$ for all $n\in\mathbb{N}$ and if $X$ is reflexive, then Condition \eqref{eq:denserange} is automatically fulfilled if $Y = \displaystyle\lim_{n\to\infty}Y_n$ (see \cite[Theorem 2.4]{ACE}). However, in the nonconforming case $X_n\not\subset X$, Condition  \eqref{eq:denserange} cannot be omitted. We give an example.

\begin{exam} Let $Z = \mathbb{R}^2$, $Y= \mathbb{R}$, $b((x_1,x_2),y) = x_1 y$. Let $Y_n = Y$, $X_{2n} = \{ (\lambda, \lambda),\lambda\in \mathbb{R}\}$ and $X_{2n+1} = \{ (\lambda, -\lambda),\lambda\in \mathbb{R}\}$,  for all $n\in\mathbb{N}$. Then $\displaystyle\lim_{n\to\infty} X_n = \{(0,0)\}$. Even though the BNB condition \eqref{eq:betabnbn} is satisfied,  \eqref{eq:denserange} fails. Moreover, \eqref{eq:pbinZ} has no solution in $X$ if $L\ne 0$.
 
\end{exam}

The following result, which is  reciprocal  to Theorem \ref{thm:bnbnonconf}, shows that the BNB condition \eqref{eq:betabnbn} and the density condition \eqref{eq:denserange} are necessary to ensure convergence properties.

\begin{thm}\label{thm:bnbnonconfrec}
Let $Z,Y$ be Banach spaces and let $b~:~Z\times Y\to \mathbb{R}$  be a continuous bilinear form. 
Let $(X_n)_{n\in\mathbb{N}}$ and $(Y_n)_{n\in\mathbb{N}}$ be sequences of finite dimensional subspaces of $Z$ and $Y$, respectively, such that  $X_n$ and $Y_n$ have the same dimension for all $n\in\mathbb{N}$.

Let us assume that $Y = {\mathop{\rm w\mbox{-}lim}\limits_{n\to\infty} Y_n}$ and that, given $L\in Y'$, for any $n\in \N$ there exists $x_n\in X_n$ satisfying \eqref{eq:pbinXn} and that the sequence $(x_n)$ converges in $Z$ as $n\to\infty$.

Then there exists $\beta>0$ such that \eqref{eq:betabnbn} holds, and, letting $X = \displaystyle\lim_{n\to\infty} X_n$, \eqref{eq:denserange} holds as well.
\end{thm}

\begin{proof}
We first reproduce the (short) proof of \cite[Proposition 2.8]{ACE} in this case. 
The assumption that, for all $L\in Y'$, there exists $x_n\in X_n$ satisfying \eqref{eq:pbinXn} proves that the condition $b(\chi,y)=0$ for all $\chi\in X_n$ implies $y=0$ whenever $y\in Y_n$, and this for all $n\in\N$. Thus 
	\[ \|y\|_{Y_n}:=\sup_{x\in X_n,\|x\|_Z=1}|b(x,y)| \]
	defines a norm on $Y_n$.  Moreover, 
	\[ |b(x,y)|\leq \|x\|_Z \|y\|_{Y_n} \mbox{ for all }x\in X_n,y\in Y_n.    \]
	We show that the set 
	  \[  {\mathcal B}:=\left\{  \frac{y}{\|y\|_{Y_n}}:n\in\N,y\in Y_n,y\neq 0      \right\}   \]
	  is bounded. For that purpose, let $L\in Y'$. By assumption there exist $x_n\in X_n$ and $C>0$ such that 
	  \[b(x_n,y)=\langle L,y\rangle \mbox{ for all }y\in Y_n\]
	  and $\|x_n\|_Z\leq C$ for all $n\in\N$. Now, for $\frac{y}{\|y\|_{Y_n}}\in {\mathcal B}$, 
	  \[  \left|\langle L,\frac{y}{\|y\|_{Y_n}}\rangle \right|=|b(x_n,y)|\frac{1}{\|y\|_{Y_n}}\leq \|x_n\|_Z\leq C.\] 
	  This shows that $\mathcal B$ is weakly bounded and thus norm-bounded. Therefore there exists $\beta>0$ such that
	  $ \|y\|_Y\leq \frac{1}{\beta}\|y\|_{Y_n}$, i.e.  such that \eqref{eq:bnbstar} holds. So far, we merely used that the sequence $(x_n)$ is bounded. In the following argument we will use that it converges.
	  
	  By Corollary \ref{cor:bnbstar} we then get that \eqref{eq:betabnbn} holds (again with the same value of $\beta$).
Let $y\in Y$ and let $(y_n)$ be a sequence such that $y_n\rightharpoonup y$ in $Y$. For $L\in Y'$, let $x$ be the limit of the converging sequence  $(x_n)$ whose existence is assumed in the theorem.
Then, passing to the limit in $b(x_n,y_n) = L(y_n)$, we get $b(x,y) = L(y)$. This proves that, letting $X = \displaystyle\lim_{n\to\infty} X_n$, the mapping $B:X \to Y'$, defined by $ \langle Bx,y\rangle_{Y',Y} = b(x,y)$ for any $(x,y)\in X\times Y$ is surjective, which implies \eqref{eq:denserange}.
\end{proof}

\begin{rem}\label{rem:vionnier}
There are very simple examples where the hypotheses of Theorem~\ref{thm:bnbnonconfrec} are satisfied but $X$  and $Y$ are not reflexive. For example, let $X=\ell^1=c'_0$,  $Y=c_0$, $b(x,y)=\displaystyle \sum_{n=0}^\infty x_n y_n$ and $X_n=\span \{ e_0,\cdots, e_n\}$ in $X$, $Y_n=\span \{ e_0,\cdots,e_n\}$ in $Y$ where $e_{nm}=\delta_{nm}$.   
\end{rem}

\section{The evolution problem}\label{sec:pbcont}

\subsection{The continuous problem with generalized time boundary conditions}

Let us now provide the continuous framework for linear parabolic problems with general time boundary conditions.

\medskip

Let $T>0$, let ${U}$ be a Hilbert space over $\R$ which is continuous and densely embedded into another Hilbert space $H$. Then there exists $C_H>0$ such that
\begin{equation}\label{eq:contemb}
\Vert v\Vert_H\le C_H\Vert v\Vert_U\hbox{ for all }v\in U.
\end{equation}
As usual we identify $H$ with a subspace of ${U}'$ by letting
\[
 \langle y,u\rangle_{U',U} = \langle y,u\rangle_H,\hbox{ for all }y\in H, \  u\in U.
\]
This yields the Gelfand triple
\[   {U}\stackrel{d}{\hookrightarrow} H\hookrightarrow {U}'.\] 

Let $V=L^2(0,T;{U})$. Thus $V'=L^2(0,T;{U}')$. Letting $\mathbb{H} = L^2(0,T;{H})$, we have a further Gelfand triple
\[   V\stackrel{d}{\hookrightarrow} \mathbb{H} \hookrightarrow V'.\]

We define the classical space $W$ associated with the Gelfand triple by  
\begin{equation}\label{eq:defspaceW}
 W = \{u\in V;\ \exists C\ge 0,
    \langle u, v'\rangle_{{\mathbb H}}\le C\Vert v\Vert_{V}\mbox{ for all }v\in C^1_c((0,T);U)\},
\end{equation}
   which enables us to define, for any $u\in W$, the element $u'\in V'$ by
   \begin{equation}\label{eq:defwpinW}
    \langle u',v\rangle_{V',V} := -\langle u, v'\rangle_{{\mathbb H}} \mbox{ for all }  v\in C^1_c((0,T);U).
   \end{equation}
In other words, we can write $W$ as follows, introducing also a Hilbert structure,
\begin{equation}\label{eq:defW}  
W=H^1(0,T;{U}')\cap L^2(0,T;{U})\hbox{ with }\Vert v\Vert_W :=\Big(\Vert v\Vert_V^2+\Vert v'\Vert_{V'}^2\Big)^{1/2}, \ v\in W.
\end{equation}

We recall that it is possible to identify $W$ with a subspace of $C([0,T];H)$ and that there exists $C_T>0$ such that
\begin{equation}\label{eq:embWinH}
\sup_{t\in[0,T]}\Vert v(t)\Vert_{H}\le C_T \Vert v\Vert_{W},\hbox{ for all }v\in W.
\end{equation}

The following integration by parts formula (\cite[III Corollary 1.1 p.106]{Sho97}) plays an important role. 
\begin{lem}\label{lem:5.1}  
One has $W\subset {\mathcal C}([0,T],H)$ and 
\begin{equation*}
 \langle v',w\rangle_{{V}',{V}} + \langle w',v\rangle_{{V}',{V}}   = \langle v(T),w(T)\rangle_H-\langle v(0),w(0)\rangle_H,\hbox{ for all }v,w\in W.
\end{equation*}
\end{lem}  
 
In the following lemma we describe the trace space of $W$. 
\begin{lem}\label{lem:5.2n}
For all $x,y\in H$ there exists $u\in W$ such that $u(0)=x$ and $u(T)=y$.  
\end{lem}
\begin{proof}  
One has $H=[ V',V]_{1/2,1/2}$, where  $[V',V]_{1/2,1/2}$ is the real interpolation space. The latter coincides with the trace space by \cite[Proposition 1.2.10]{Lunardi}. 
Hence there exist  $w\in W$ such that $w(0)=x$ and $v\in W$ such that $v(T)=y$. Then the function $u(s) = \frac 1 T ((T-s) w(s) + s v(s))$ satisfies  $u(0)=x$ and $u(T)=y$.
\end{proof}

Let ${\mathcal A}\in {\mathcal L}(V,V')$ be defined  by
\[ \langle {\mathcal A}v,w\rangle_{V',V}=\int_0^T a(t,v(t),w(t)){\rm d}t \mbox{ for all }v,w\in V=L^2(0,T;{U}), \]
 where  $a:[0,T]\times {U}\times {U}\to \R$ is a function  satisfying
 \begin{itemize}
 	\item[a)] $a(t,\cdot,\cdot):{U}\times {U}\to\R$ is bilinear for all $t\in [0,T]$;
 	\item[b)] $a(\cdot, v,w):[0,T]\to\R$ is measurable for all $v,w\in {U}$;
 	\item[c)] $|a(t,v,w)|\leq M\|v\|_{{U} }\|w\|_{{U}}$ for a.e. $t\in (0,T)$ and for all $v,w\in  {U}$ (we take $M\ge 1$ for the sake of simplicity);
 	\item[d)] $a(t,v,v)\geq \alpha \|v\|_{{U}}^2$ for all $t\in [0,T]$, $v\in {U}$ and some $\alpha>0$.   
 \end{itemize}

Let $\Phi~:~H\to H$ be a linear contraction (which means that  $\Vert\Phi v\Vert_H\le \Vert v\Vert_H$ for all $v\in H$). In the centre of this article is the following problem. Given $f\in V'$, $\xi_0\in H$,
\begin{equation}\label{eq:5.1}
\hbox{find }u\in W\hbox{ such that } u'+{\mathcal A}u=f\hbox{ and }u(0)-\Phi u(T)=\xi_0.
\end{equation}

\begin{thm}\label{thm:5.3}
 For all  $f\in V'$, $\xi_0\in H$, Problem \eqref{eq:5.1} has a unique solution.
\end{thm}

For $\Phi=0$,  Problem \eqref{eq:5.1} is an initial value problem whose well-posedness is due to J-L. Lions  \cite[Théorème 1.1 p.46]{Lio61}, see also \cite[III Proposition 2.3 p.112]{Sho97} or \cite[XVIII.3 Th\'eor\`eme 2]{DL}. For $\Phi = {\rm Id}$, $\xi_0 = 0$, we find time-periodic boundary conditions. If $\Phi$ is an arbitrary contraction and $a$ is symmetric, then Theorem \ref{thm:5.3} is proved in \cite[II. Proposition 2.3]{Sho97}, where selfadjointness is used in an essential way. The general case has been proved recently in \cite{RDV1} in the framework of a theory of derivations. We obtain another proof as a byproduct of our estimates which we give in Section \ref{sub:4.3}.

\begin{exam}
  The following example is the prototype of parabolic problems:
  \begin{itemize}
   \item The considered steady problem is the homogeneous Dirichlet   problem, which corresponds to $U = H^1_0(\Omega)$, where $\Omega\subset \mathbb{R}^d$  is an open bounded set, and $\Vert u\Vert_U = \Vert \nabla u\Vert_{L^2(\Omega)^d}$.

   \item  $a(t,\cdot,\cdot):{U}\times {U}\to\R$ is defined  by  $a(t,u,v) = \int_\Omega \nabla u(x)\cdot\nabla v(x) {\rm d}x$.
   \item We let $H = L^2(\Omega)$. 
   \item If $\Phi = 0$ is selected, then the above hypotheses on the form $a$ hold with $M=\alpha = 1$.
  \end{itemize}

\end{exam}

\begin{rem}
 The coercivity condition d) may be replaced by the more general condition 
 \begin{itemize}
 	\item[d')] $a(t,v,v) + \lambda\Vert v\Vert_{H}^2\geq \alpha \|v\|_{{U}}^2$ for all $v\in {U}$ and all $t\in [0,T]$,
 \end{itemize} 
 where $\lambda\ge 0$ and $\alpha>0$ are fixed constants. In that case, we have to assume that $\Vert\Phi\Vert\le e^{-\lambda T}$. In fact, the following rescaling allows one to reduce the problem to the case $\lambda = 0$: let $\widetilde{\Phi}= e^{\lambda T}\Phi$, $\widetilde{a}(t,u,v) = a(t,u,v)+\lambda\langle u,v\rangle_H$. Then condition d) is verified for $\widetilde{a}$. Let $f\in V'$. Define $\widetilde{f}$ by $\widetilde{f}(t) = e^{-\lambda t}f(t)$. Then by Theorem \ref{thm:5.3} there exists a unique $\widetilde{u}$  such that $\widetilde{u}'+\widetilde{\mathcal A}\widetilde{u}=\widetilde{f}\hbox{ and }\widetilde{u}(0)-\widetilde{\Phi} \widetilde{u}(T)=\xi_0$, where $\widetilde{\mathcal A} v = {\mathcal A} v + \lambda v$, $v\in V$. Let $u(t) = e^{\lambda t}\widetilde{u}(t)$. Then $u$ is the unique solution of Problem \eqref{eq:5.1}.
 \end{rem}

 The purpose of this article is to define finite dimensional approximations of Problem \eqref{eq:5.1} whose solutions converge to the solution of \eqref{eq:5.1}. In the case $\Phi = 0$ a semi-discretization (i.e. a mere discretization in the space $U$ and not in time) is given in \cite{DL}. Our point is to discretize the time and space variables simultaneously.

\subsection{Formulation in an extended framework}\label{sec:extended}
In order to define a discrete approximation it will be convenient to give an equivalent variational formulation of  Problem \eqref{eq:5.1}.

We denote by $\mathcal{B}([0,T];H)$ the space of all bounded functions from $[0,T]$ to $H$ and we define the space
 \begin{multline}\label{eq:defG}
  G = \mathcal{B}([0,T];H)\cap V = \{ w\in\mathcal{B}([0,T];H): w(t)\in U\hbox{ for a.e. }t\in (0,T), \\\
  \hbox{ }t\mapsto w(t)\hbox{ is measurable from }(0,T)\to U\hbox{ and}\int_0^T\Vert w(t)\Vert_U^2{\rm d}t<\infty\}
 \end{multline}
\begin{lem}
 $G$ is a Banach space equipped with the norm
 \begin{equation}\label{eq:defnorG}
   \Vert w\Vert_G^2 = \sup_{t\in[0,T]}\Vert w(t)\Vert_H^2 + \Vert w\Vert_V^2.
\end{equation}
\end{lem}
\begin{proof}
 Let $(u_n)$ be a Cauchy sequence in $G$. Then there exists $u\in \mathcal{B}([0,T];H)$ such that $u_n(t)\to u(t)$ in $H$ uniformly on $[0,T]$. Since the sequence $(u_n)$ is a Cauchy sequence in $V$, there exists $v\in V$ such that $u_n\to v$ in $V$. Consequently, there exists a subsequence of $(u_n)$ converging almost everywhere to $v$ in $U$. This shows that $u\in G$, $u=v$ a.e. and that $\Vert u - u_n\Vert_G\to 0$ as $n\to\infty$.
\end{proof}
Notice that $W\subset G$.

Let us introduce the following spaces:

\begin{itemize}
 \item We define
 
 \begin{equation}\label{eq:defyz}
 	Z = V'\times G \mbox{ and }Y = V\times H,
 \end{equation}
 with norms $\Vert(z_1,z_2)\Vert_Z^2 := \Vert z_1\Vert_{V'}^2 + \Vert z_2\Vert_{G}^2 $ and $\Vert(y_1,y_2)\Vert_Y^2 := \Vert y_1\Vert_{V}^2 + \Vert y_2\Vert_{H}^2$.
 \item We define the space $X \subset Z$ by 
 \begin{equation}\label{eq:defhatw}
  X = \{ (u',u)\in Z, u\in W\}.
 \end{equation}
 \item We define the form $b~:~Z\times Y\to \mathbb{R}$ by
 \begin{multline}\label{eq:defb} 
 b((z_1,z_2),(y_1,y_2)) = \langle z_1 + \mathcal{A} z_2,y_1\rangle_{V',V} + \langle z_2(0) - \Phi z_2(T),y_2\rangle_{H}\\
 \hbox{for all }(z_1,z_2)\in Z\hbox{ and } (y_1,y_2)\in Y.
 \end{multline}
 Since, for any $(z_1,z_2)\in Z$ and $(y_1,y_2)\in Y$, we have
\[
  |b((z_1,z_2),(y_1,y_2))|\le \Vert z_1 + \mathcal{A} z_2\Vert_{V'}\Vert y_1\Vert_V +  \Vert z_2(0) - \Phi z_2(T)\Vert_{H}\Vert y_2\Vert_H, 
\]
we get
\[
  |b((z_1,z_2),(y_1,y_2)))|\le\Big( (\Vert z_1 \Vert_{V'} + M\Vert z_2\Vert_{V})^2 + (\Vert z_2(0) \Vert_{H}+\Vert  z_2(T)\Vert_{H})^2\Big)^{1/2}\Vert (y_1,y_2)\Vert_Y, 
\]
which yields, since $M\ge 1$,
\[
  |b(z,y)|\le2 M \Vert z\Vert_Z\Vert y\Vert_Y.
\]
 \item We define $L~:~Y\to \mathbb{R}$ by

\begin{equation}\label{eq:defL} 
 L(y)= \langle f, y_1\rangle_{V',V}+\langle \xi_0, y_2\rangle_{H}\hbox{ for all }(y_1,y_2)\in Y,
\end{equation}
where $f\in V'$, $\xi_0\in H$ are the given data in Problem~\ref{eq:5.1}.
\end{itemize}

We then observe that Problem \eqref{eq:5.1} is equivalent to the variational problem: 
\begin{equation}\label{eq:pbcontinZ}
\hbox{find }\widehat{x}=(u',u)\in X\hbox{ such that }b(\widehat{x},y) = L(y)\hbox{ for all }y\in Y.
\end{equation}
 Defining suitable finite dimensional subspaces $X_n$ of $Z$ and $Y_n$ of $Y$ we will obtain the Euler scheme in Section \ref{sec:eulerex} and the discontinuous Galerkin scheme in Section \ref{sec:dg}. In these sections it is proved that $X\subset \displaystyle\lim_{n\to\infty} X_n$ (see Theorems \ref{thm:3.9} and \ref{thm:3.9dg}).  Note that $X\subset \displaystyle\lim_{n\to\infty} X_n$ and the BNB condition imply the inf-sup condition \eqref{eq:betabnb} by Lemma~\ref{lem:infsup}. Nonetheless the inf-sup condition could also  be proved using Lemma~\ref{lem:suffbnb}.

 In order to apply Theorem \ref{thm:bnbnonconf}, it is also important to check that $X$ is a closed space and fulfills Condition \eqref{eq:denserange} which we do in the following lemma.
 
 \begin{lem}\label{lem:xhasdenserange}
  The space $X$ defined by \eqref{eq:defhatw} and  \eqref{eq:defyz} is a closed subspace of $Z$ and fulfills Condition \eqref{eq:denserange} with respect to the bilinear form $b$.
 \end{lem}
 \begin{proof}
 Let $(w_n)_{n\in \mathbb{N}}$ be a sequence of elements of $W$ such that $(w'_n)_{n\in{\mathbb N}}$ and $(w_n)_{n\in {\mathbb N}}$ are converging sequences in $V'$ and $G$  to some functions $r\in V'$ and $w\in G$, respectively. Then for all $v\in C^1_c((0,T);U)$, we have
 \[
  \langle w'_n,v\rangle_{V',V} = -\langle v',w_n\rangle_{\mathbb{H}}\mbox{ for all }n\in {\mathbb N}.
 \]
 Letting $n\to\infty$ in the above relation, we get, since the convergence in $G$ implies that in $V$ and therefore that in $\mathbb{H}$,
 \[
   \langle r,v\rangle_{V',V} = -\langle v',w\rangle_{\mathbb{H}}.
 \]
 Using \eqref{eq:defspaceW}-\eqref{eq:defwpinW} we deduce that $w\in W$ and $r = w'$. We have shown  that $X$ is closed.

 \medskip

  Turning to the proof of \eqref{eq:denserange}, let $y := (w,z)\in Y = V\times H$  such that
  \[
   b(x,y)=0\mbox{ for all }x\in X.
  \]
Let us prove that $y = 0$. In view of the definitions \eqref{eq:defhatw} of $X$ and \eqref{eq:defb} of $b$, the preceding relation means that
\[
  \langle v' +{\mathcal A} v,w\rangle_{V',V} + \langle v(0)-{\Phi} v(T)),z\rangle_H = 0\mbox{ for all }v\in W.
\]
We first take $v =r\in C^1_c(]0,T[;U)$. Then we obtain 
\[
 \langle r',w\rangle_{V',V}= \langle r',w\rangle_{\mathbb{H}} = -\langle {\mathcal A} r,w\rangle_{V',V} = -\langle {\mathcal A}^* w,r\rangle_{V',V}.
\]
In view of the definitions \eqref{eq:defspaceW} of $W$ and \eqref{eq:defwpinW} of $w'$, this shows that  $w'\in V'$ with $w' = {\mathcal A}^* w$. Hence $w\in W$.
By  Lemma \ref{lem:5.1} we get 
\[
 \langle v',w\rangle_{V',V} = -\langle w',v\rangle_{V',V}+\langle w(T),v(T)\rangle_{H}-\langle w(0),v(0)\rangle_{H}.
\]
Then
\[
  \ \langle w(T),v(T)\rangle_{H}-\langle w(0),v(0)\rangle_{H}+ \langle v(0)-{\Phi} v(T),z\rangle_H=0 \mbox{ for all }v\in W,
\]
which means that
\[
 \ \langle w(T)-\Phi^* z,v(T)\rangle_{H}+\langle z-w(0),v(0)\rangle_{H}=0 \mbox{ for all }v\in W.
\]
Lemma \ref{lem:5.2n} implies that $z = w(0)$ and $w(T) = \Phi^* w(0)$.
Hence, since $w' - {\mathcal A}^* w = 0$, we get
\[
\frac 1 2 (\Vert \Phi^* w(0)\Vert_H^2 -\Vert w(0)\Vert_H^2) - \langle {\mathcal A}^* w,w\rangle_{V',V} =  \langle w' - {\mathcal A}^* w,w\rangle_{V',V} =0.
\]
The coercivity property $ \langle {\mathcal A}^* w,w\rangle_{V',V} \ge \alpha \Vert w\Vert_V^2$ and the property $\Vert\Phi\Vert\le 1$ imply $w=0$, and therefore $z = w(0) = 0$.

This finishes the proof of Condition \eqref{eq:denserange}.
 \end{proof}

\section{Euler time-space approximation}\label{sec:eulerex}

In this section we keep the framework of Section \ref{sec:pbcont}. In particular, ${\Phi} :H\to H$ is linear and $\Vert\Phi\Vert\le 1$. Moreover, $f\in V'=L^2(0,T;U')$ and $\xi_0\in H$ are given. 

The aim of this section is to define a finite dimensional approximation of Problem \eqref{eq:5.1} and to show that it fits into the framework of Section \ref{sec:NCBNB}. 

\subsection{Description of the Euler schemes}\label{sub:4.1}
  
 Let $U_n$ be a finite dimensional subspace of $U$. For the moment, the lower index $n$ is fixed and used for designing a fixed finite dimensional space.
 
 Let $N\in\mathbb{N}\setminus\{0\}$ and define $k = \frac {T}{N}$. 
 For all $m=0,\ldots,N-1$, $A^{(m)}~:~U_n\to U_n$ denotes the coercive linear operator given by
\[
\langle A^{(m)} u,v\rangle_{U}  = \frac 1 {k}\int_{mk}^{(m+1) k} a(t,u,v) {\rm d}t\mbox{ for all } u,v\in U_n
\]
and  $f^{(m)}\in U'$ is defined by
\[
 f^{(m)} = \frac 1 {k}\int_{mk}^{(m+1) k} f(t) {\rm d}t.
\]

 The explicit Euler scheme consists in seeking $N+1$ elements of $U_n$, denoted by $(w^{(m)})_{m=0,\ldots,N}$, such that
\begin{equation}\label{eq:schemeiniex}
\langle w^{(0)} - \Phi w^{(N)} , u\rangle_H  =  \langle \xi_0 , u\rangle_H\mbox{ for all }u\in U_n
\end{equation}
and 
\begin{equation}\label{eq:schememex}
\langle \frac{w^{(m+1)}-w^{(m)}}{k}, u\rangle_{ H } + \langle A^{(m)} w^{(m)},u\rangle_{U} =  \langle f^{(m)},u\rangle_{U',U}\mbox{ for all }m=0,\ldots,N-1 \mbox{ and }u\in U_n.
\end{equation}

Note that, if $\Phi\equiv 0$, the scheme is the usual explicit scheme, and the existence of a solution to \eqref{eq:schememex} is clear. In the general case, a linear system involving $w^{(N)}$ must be solved, and its invertibility is proved by Theorem \ref{lem:estimnoruTHthetaex}, under a condition on $k$.  The invertibility of this system is not true in general (consider the case $H = U$, $N=1$, $A^{(0)} = 2 {\rm Id}$, $k=T=1$, $\Phi = -{\rm Id}$).

 The implicit Euler scheme consists in seeking $N+1$ elements of $U_n$, denoted by $(w^{(m)})_{m=0,\ldots,N}$, such that \eqref{eq:schemeiniex} holds and 
\begin{equation}\label{eq:schememim}
\langle \frac{w^{(m+1)}-w^{(m)}}{k}, u\rangle_{ H } + \langle A^{(m)} w^{(m+1)},u\rangle_{U} =  \langle f^{(m)},u\rangle_{U',U}\mbox{ for all }m=0,\ldots,N-1\mbox{ and }u\in U_n.
\end{equation}

Both schemes are therefore particular cases of the so-called $\theta-$scheme, given by  \eqref{eq:schemeiniex} and, for a given $\theta\in [0,1]$,
\begin{multline}\label{eq:schememextheta}
 \langle \frac{w^{(m+1)}-w^{(m)}}{k}, u\rangle_{ H } + \langle A^{(m)} (\theta w^{(m+1)}+ (1-\theta) w^{(m)}),u\rangle_{U} =  \langle f^{(m)},u\rangle_{U',U}\\
\mbox{ for all }m=0,\ldots,N-1\mbox{ and }u\in U_n .
\end{multline}
 
The case $\theta=0$ gives the explicit Euler scheme, the case $\theta = 1$ gives the implicit Euler scheme, and the case $\theta = \frac 1 2$ leads to the Crank-Nicolson scheme.

In the following, we will associate to each $(w^{(m)})_{m=0,\ldots,N}$ a step function and its discrete derivative, and we will show that this discrete solution  approximates in $Z$ (defined in \eqref{eq:defyz}) the solution of Problem \eqref{eq:pbcontinZ} (i.e. the variational formulation of Problem \eqref{eq:5.1}). For this purpose, we now define the space $W_n\subset G$ (see \eqref{eq:defG}) of all functions $w~:~[0,T]\to U_n$ for which there exist $N+1$ elements of $U_n$, denoted by $(w^{(m)})_{m=0,\ldots,N}$, such that
 \begin{multline}\label{eq:schememtheta}
   w(t) = \theta w^{(m+1)}+ (1-\theta) w^{(m)}\hbox{ for all }t\in (mk, (m+1)k),\hbox{ for all }m=0,\ldots,N-1\\
   \hbox{ and }w(mk) = w^{(m)}\hbox{ for all }m=0,\ldots,N. 
 \end{multline}
We observe that the space $W_n$ is isomorphic to $U_n^{N+1}$.

Each element of $W_n$ defines a unique element of $V$, but this identification is not injective, because of the initial or the final values. To keep the notation simple, we do not distinguish the notation for the everywhere defined $w\in W_n$ and the element of $V_n\subset V$ it defines, where $V_n$ is the subspace of $V$ defined by
\begin{equation}\label{eq:defVn}
V_n := \{v :(0,T)\to U_n: = v \mbox{ is constant on }(mk,(m+1)k)\mbox{ for each }m=0,\cdots, N-1 \}.
\end{equation}
We define the discrete derivative of $w\in W_n$ as follows:
 \begin{equation}\label{eq:defqnevolstrex}
 \partial w(t) = \frac{w^{(m+1)}-w^{(m)}}{k},\  t\in (mk, (m+1)k)\hbox{ for all }m=0,\ldots,N-1.
 \end{equation}
 But we will have to define this discrete derivative as an element of $V'$ in an unusual way. Let us denote by $P^V_n$ the orthogonal projection on $V_n$ in $V$. Note that this can be explicitely written in the following way:
\begin{equation}\label{eq:projvproju}
  P^V_n v(t) = \frac 1 k\int_{mk}^{(m+1)k} P^U_n v(s) {\rm d}s, \ t\in (mk, (m+1))\mbox{ for all }v\in V,
\end{equation}
 where $P^U_n$ denotes the orthogonal projection on $U_n$ in $U$. We now define  the discrete derivative of an element $w\in W_n$ as an element $\widehat{\partial}w$  of $V'$ in the following way:
 \begin{equation}\label{eq:defdhatfromd}
  \langle \widehat{\partial}w,v\rangle_{V',V} = \langle \partial w,P^V_n v\rangle_{\mathbb{H}}\mbox{ for all }v\in V.
\end{equation}
Using \eqref{eq:projvproju}, one can check that
 \[
  \langle \widehat{\partial} w,v\rangle_{V',V} = \int_{0}^{T} \langle  \partial w(t),P^U_n v(t)\rangle_{H}{\rm d}t\mbox{ for all }w\in W_n,v\in V. 
 \]

Note that $\widehat{\partial}w$ is different from $\partial w$ seen as an element of $V'$ through the embedding of $\mathbb{H}$ in $V'$, but that we have
\[
   \langle \widehat{\partial}w, v\rangle_{V',V} =\langle {\partial} w, v\rangle_{\mathbb{H}} \mbox{ for all } w,v\in W_n,
\]
as well as
\begin{equation}\label{eq:normeDUVprime}
 \Vert \widehat{\partial}w\Vert_{V'} = \sup_{v\in V,\Vert v\Vert_V = 1}\langle \partial w,P^V_n v\rangle_{\mathbb{H}} = \sup_{v\in V_n,\Vert v\Vert_V = 1}\langle \partial w,v\rangle_{\mathbb{H}}\mbox{ for all }w\in W_n .
\end{equation}
The definition \eqref{eq:defdhatfromd} of $\widehat{\partial}$ will imply Condition \eqref{eq:hypxun} in Lemma \ref{lem:suffbnb}. This in turn will be used to establish the BNB condition.
	
In order to put this setting into the framework of Sections \ref{sec:NCBNB} and \ref{sec:pbcont}, we define the space $X_n\subset Z := V'\times G$ by
\begin{equation}\label{eq:defxhat}
X_n = \{(\widehat{\partial} w,w): w\in W_n\}.
\end{equation}
The mapping $W_n\to X_n$, $w\mapsto (\widehat{\partial} w,w )$ is then bijective. We now define the space $Y_n \subset Y := V\times H$ by
\begin{equation}\label{eq:defY}
Y_n = V_n\times U_n.
\end{equation}
Then $Y_n$ has the same dimension ${\rm dim}(U_n)^{N+1}$ as $W_n$ and $X_n$.

We recall that the bilinear form $b~:~Z\times Y\to \mathbb{R}$ is defined by \eqref{eq:defb} and $L~:~Y\to \mathbb{R}$ is defined  by \eqref{eq:defL}.
\begin{lem}\label{lem:schemeexpabs} Let $(w^{(m)})_{m=0,\ldots,N}$ be a sequence of elements of $U_n$. Let $w\in W_n$ be given by \eqref{eq:schememtheta}. Let $x = (\widehat{\partial} w,w)$. The following assertions are equivalent:
\begin{enumerate}
 \item[(i)] The family  $(w^{(m)})_{m=0,\ldots,N}$ satisfies \eqref{eq:schemeiniex} and \eqref{eq:schememextheta}.
\item[(ii)] The element $x$ is solution of the following problem:
\begin{equation}\label{eq:schemeinZY}
\hbox{find }x\in X_n\hbox{ such that } b(x,y) = L(y)\hbox{ for all }y\in Y_n.
\end{equation}
\end{enumerate}
\end{lem}
\begin{proof} 
(i)$\Longrightarrow$(ii).

Let $y :=(v,z)\in Y_n$, with $v\in V_n$ defined from the constant values $(v^{(m)})_{m=1,\ldots,N}$. Letting $u = k v^{(m)}$ in \eqref{eq:schememex} and $u = z$  in \eqref{eq:schemeiniex}, adding the result provides \eqref{eq:schemeinZY}.

(ii)$\Longrightarrow$(i).

For $u\in U_n$, we select $y :=(v,z)\in Y_n$, with $v\in V_n$ defined from $(v^{(m)})_{m=0,\ldots,N}$ with $v^{(m)} = 0$ for $m=0,\ldots N-1$ and $z= u$ in \eqref{eq:schemeinZY}. We recover  \eqref{eq:schemeiniex}. Now selecting $z=0$ and $v^{(m)} = u$ for any $m=0,\ldots,N-1$,  $v^{(m')} = 0$ for $m'\ne m$, we recover  \eqref{eq:schememextheta}.
 
\end{proof}

It is remarkable that the continuous problem is obtained by injecting $W$ in $Z$, and the discrete problem by injecting $W_n$ into $Z$. This will allow us to use the results of Section \ref{sec:NCBNB}.
For that we have to establish the BNB condition \eqref{eq:betabnbn}.

\subsection{The BNB condition}\label{sub:4.2}
%
%
%
We keep the setting of Subsection~\ref{sub:4.1}.  
 An important remark is the following: since all norms are equivalent in finite dimensional spaces, there exists a constant $\mu_n$ such that
 \begin{equation}\label{eq:defmuh}
   \Vert u\Vert_U \le \mu_n  \Vert u\Vert_H\mbox{ for all } u\in U_n.
 \end{equation}
We choose the best constant; that is,  $\mu_n=\max\{ \frac{\|u\|_{U}}{\|u\|_H}:u\in U_n\setminus \{0\}\}$. 

One should be aware that $\mu_n$ strongly depends on $U_n$, and generally tends to infinity as $U_n$  approximates $U$ (in the case of a standard elliptic problem, and in the case where $h$ is a mesh size, then $\mu_n$ behaves as $1/h$, see \cite[Remark 1.143 p.77]{EG04}).

 But recall that, for the moment, $U_n$ is fixed.
 
 Let us show the following preliminary lemma.
 
\begin{lem}\label{lem:normhnormVp}The following holds.
\begin{equation}\label{eq:normhnormVp}
 \Vert \partial w\Vert_{\mathbb{H}}\le k\mu_n \Vert \widehat{\partial} w\Vert_{V'}\mbox{ for all }w\in W_n.
\end{equation}
\end{lem}
  
\begin{proof} 
 Let $v\in V_n$ with  $v^{(m)} = w^{(m+1)} - w^{(m)}$, $m=0,\ldots,N-1$. Then
\begin{multline*}
	 \sum_{m=0}^{N-1}k\Vert w^{(m+1)} - w^{(m)}\Vert_H^2 =k\langle \widehat{\partial} w, v\rangle_{V',V}\le k\Vert \widehat{\partial} w\Vert_{V'}\Vert v\Vert_{V} \\
	\le  \frac {k^2 \mu_n^2} 2 \Vert \widehat{\partial} w\Vert_{V'}^2 +\frac 1 {2  \mu_n^2} \Vert v\Vert_{V}^2 \\
	\le \frac {k \mu_n^2} 2 \Vert \widehat{\partial} w\Vert_{V'}^2 + \frac 1 {2 } \sum_{m=0}^{N-1} k\Vert w^{(m+1)} - w^{(m)}\Vert_H^2.
	\end{multline*}
This yields
\begin{equation}\label{eq:estimsumdif}
 \sum_{m=0}^{N-1} k\Vert w^{(m+1)} - w^{(m)}\Vert_H^2 \le  k^2 \mu_n^2   \Vert \widehat{\partial} w\Vert_{V'}^2,
\end{equation}
which yields \eqref{eq:normhnormVp}.
\end{proof}

We  now state and prove the BNB property for this scheme under a CFL-condition. This type of condition had at first been introduced by Courant, Friedrichs and Lewy in \cite{CFL}. If $\mu_n\le$ const $n$ (as for the finite element method with $h=1/n$), then it says that $N\ge $const$\cdot n^2$.
Recall that $b:Z\times Y\to\mathbb{R}$ is the form defined by \eqref{eq:defb} where $Z = V'\times G$, $Y = V\times H$ with the natural norms. The subspaces $X_n$ of $Z$ and $Y_n$ are given by \eqref{eq:defxhat} and \eqref{eq:defY}. In the following key lemma the essential hypothesis is that the time step $k = \frac T N$ is small enough. Recall that $\mu_n$ is the constant of \eqref{eq:defmuh}.

\begin{lem}\label{lem:estimnoruTHthetaex}  
Let $\theta\in [0,\frac 1 2)$. 
There exists a constant $\beta\in (0,2M]$, only depending on $\alpha$, $M$, $T$, $\theta$ and $C_H$, such that the following holds.
If $N$ is so large that
	\begin{equation}\label{eq:cfltheta}
	 k\le \frac  {\alpha^2 }{24\, \mu_n^2 M^3 (\frac 1 2 - \theta)},
	\end{equation}
then 
	\begin{equation}\label{eq:bnbeulerextheta}
	 \sup_{\substack{y\in Y_n\\ \Vert y\Vert_Y = 1}} b(x,y) \ge \beta \Vert x\Vert_Z \hbox{ for all }x\in X_n.
	\end{equation}
 \end{lem}
\begin{proof} 
We have, for $0\le m < m'\le N$,
\begin{multline}\label{eq:ineqestimtheta}
 \int_{mk}^{m'k} \langle \partial w(t), w(t)\rangle_H{\rm d}t = \sum_{p=m}^{m'-1} k \langle \frac {w^{(p+1)} - w^{(p)}} k,\theta w^{(p+1)}+(1-\theta)w^{(p)}\rangle_H \\
 = \frac 1 2 \Vert w^{(m')}\Vert_H^2 +(\theta-\frac 1 2)\sum_{p=m}^{m'-1} \Vert w^{(p+1)} - w^{(p)}\Vert_H^2 -\frac 1 2 \Vert w^{(m)}\Vert_H^2.
\end{multline}
Letting $m' = N$ in the previous equation provides
\[
 \frac 1 2 \Vert w(T)\Vert_H^2  +   \Vert \widehat{\partial} w \Vert_{V'}\Vert w\Vert_{V}\ge \frac 1 2 \Vert w^{(m)}\Vert_H^2\hbox{ for }m=0,\ldots,N.
\]
Using the Young inequality and the triangle inequality, this gives 
\begin{equation}\label{eq:estimusupbynorminZtheta} 
\sup_{t\in[0,T]}\Vert w(t)\Vert_H  \le  \Big( \Vert \widehat{\partial} w \Vert_{V'}^2 + \Vert  w\Vert_V^2  + \Vert w(T)\Vert_H^2\Big)^{1/2}\mbox{ for all }w\in W_n.
\end{equation}
From \eqref{eq:ineqestimtheta}, we get, letting $m= 0$ and $m'=N$,
\[
 \frac 1 2 (\Vert w(T)\Vert_H^2  -\Vert w(0)\Vert_H^2)  =  \langle \widehat{\partial} w, w\rangle_{V',V}+ (\frac 1 2-\theta)\sum_{p=0}^{N-1} \Vert w^{(p+1)} - w^{(p)}\Vert_H^2,
\]
and therefore, applying \eqref{eq:normhnormVp} in Lemma \ref{lem:normhnormVp},
\[
 \frac 1 2 (\Vert w(T)\Vert_H^2  -\Vert w(0)\Vert_H^2)  \le  \langle \widehat{\partial} w, w\rangle_{V',V}+ (\frac 1 2-\theta)k \mu_n^2   \Vert \widehat{\partial} w\Vert_{V'}^2.
\]

Now using assumption \eqref{eq:cfltheta}, we obtain

\begin{equation}\label{eq:ineqestimthetab}
 \frac 1 2 (\Vert w(T)\Vert_H^2  -\Vert w(0)\Vert_H^2)  \le  \langle \widehat{\partial} w, w\rangle_{V',V}+ \frac {\alpha^2 }{24\,M^3 }   \Vert \widehat{\partial} w\Vert_{V'}^2.
\end{equation}

Letting  $m=0$ in \eqref{eq:ineqestimtheta}, we obtain
\[
 \frac k 2 \Vert w^{(m')}\Vert_H^2  +  k \Vert \widehat{\partial} w \Vert_{V'}\Vert w\Vert_{V}\ge \frac k 2 \Vert w(0)\Vert_H^2\hbox{ for }m'=0,\ldots,N.
\]
Summing the preceding equation on $m'=0,\ldots,N-1$ and multiplying by 2, we get
\[
  k  \sum_{m'=0}^{N-1}\Vert w^{(m')}\Vert_H^2  +  2T \Vert \widehat{\partial} w \Vert_{V'}\Vert w\Vert_{V}\ge  T \Vert w(0)\Vert_H^2.
\]
We now remark that
\begin{multline*}
 \Vert w\Vert_{\mathbb{H}}^2 = \sum_{m=0}^{N-1}k \Vert\theta w^{(m+1)}+(1-\theta)w^{(m)}\Vert_H^2\\
 \ge (1 - 2\theta)^2\sum_{m=0}^{N-1}k \Vert w^{(m)}\Vert_H^2  + \theta(1-2\theta)k(\Vert w(0)\Vert_H^2- \Vert w(T)\Vert_H^2).
\end{multline*}
The two preceding relations give
\[
 \frac {C_H^2}{(1 - 2\theta)^2}\Vert w\Vert_{V}^2 + \frac {\theta}{1 - 2\theta} k(\Vert w(T)\Vert_H^2- \Vert w(0)\Vert_H^2)
    +  T (\Vert \widehat{\partial} w \Vert_{V'}^2 + \Vert w\Vert_{V}^2)\ge T \Vert w(0)\Vert_H^2.
\]
Using \eqref{eq:ineqestimthetab} and $k\le T$, we deduce

\begin{multline*}
 \frac {C_H^2}{(1 - 2\theta)^2}\Vert w\Vert_{V}^2 + \frac {\theta}{1 - 2\theta} T\Big(2\langle \widehat{\partial} w, w\rangle_{V',V}+ \frac {\alpha^2 }{12\,M^3 }      \Vert \widehat{\partial} w\Vert_{V'}^2\Big)
    +  T (\Vert \widehat{\partial} w \Vert_{V'}^2 + \Vert w\Vert_{V}^2)\ge T \Vert w(0)\Vert_H^2.
\end{multline*}

This leads to the existence of $\ctel{cst:ineqtheta}>0$, only depending on $T$, $C_H$, $\alpha$, $\theta$ and $M$ such that

\[
 \frac  {\alpha^2 }{24\, M^3}(\Vert \widehat{\partial} w \Vert_{V'}^2 + \Vert w\Vert_{V}^2)\ge \cter{cst:ineqtheta} \Vert w(0)\Vert_H^2.
\]
In addition to \eqref{eq:ineqestimthetab}, we obtain
\begin{multline*}
  \langle \widehat{\partial}w,w\rangle_{V',V}+  \frac  {\alpha^2 }{12\, M^3}(\Vert w \Vert_{V}^2+\Vert \widehat{\partial}w \Vert_{V'}^2)\ge \frac {1} 2 (\Vert w(T)\Vert_H^2 - \Vert w(0)\Vert_H^2)+\cter{cst:ineqtheta}\Vert w(0)\Vert_H^2,\\
  \hbox{ for all } w\in W_n.
\end{multline*}
  We then obtain \eqref{eq:condlim} with
$\omega = \frac 1 2$ and $\delta = \min(\frac 1 2 ,\cter{cst:ineqtheta})$ (note that $\omega$ and $\delta$ only depend on $\alpha$, $M$, $T$, $\theta$ and $C_H$).
We can now apply Lemma  \ref{lem:suffbnb}, letting $\widehat{X} = \{(\widehat{\partial} w,w,w(0),w(T)): w\in W_n\}$, since the quantities $\widehat{X}_i$ defined in the lemma satisfy $\widehat{X}_2 = V_n$ and $\widehat{X}_3 = U_n$. We get the existence of $\widehat{\beta}$,  only depending on $\alpha$, $M$, $T$, $\theta$ and $C_H$, such that
\[
	 \sup_{\substack{y\in Y_n\\ \Vert y\Vert_Y = 1}} b((\widehat{\partial} w,w),y) \ge \widehat{\beta} \Big( \Vert \widehat{\partial} w \Vert_{V'}^2 + \Vert  w\Vert_V^2 +\Vert w(0)\Vert_H^2 + \Vert w(T)\Vert_H^2\Big)^{1/2}\hbox{ for all }w\in W_n.
\]
Using \eqref{eq:estimusupbynorminZtheta} we obtain   $\beta$ (with the same dependencies) such that \eqref{eq:bnbeulerextheta} holds. 
\end{proof}

\begin{lem}\label{lem:estimnoruTHthetademi}  
 
Let $\theta=\frac 1 2$. Then there exists $\beta >0$ such that 
	\begin{equation}\label{eq:bnbeulerimthetademi}
	 \sup_{\substack{y\in Y_n\\ \Vert y\Vert_Y = 1}} b(x,y) \ge \beta \Vert x\Vert_Z \hbox{ for all }x\in X_n.
	\end{equation}
with the following dependencies.

{\bf Case $\Vert\Phi\Vert < 1$.} 

Then $\beta$  may be chosen only depending on  $T$,  $\alpha$, $M$ and  $\Vert\Phi\Vert$.

{\bf Case $\Vert\Phi\Vert = 1$.} 

Assuming that $k \mu_n \le C$, for some given $C>0$. Then 
$\beta$ may be chosen only depending on $\alpha$, $M$, $T$ and  $C$.

 \end{lem}
\begin{proof} 
We have, for $0\le m < m'\le N$,
\begin{multline}\label{eq:ineqestimthetademi}
 \int_{mk}^{m'k} \langle \partial w(t), w(t)\rangle_H{\rm d}t = \sum_{p=m}^{m'-1} k \langle \frac {w^{(p+1)} - w^{(p)}} k, \frac 1 2(w^{(p+1)}+w^{(p)})\rangle_H \\
 = \frac 1 2 \Vert w^{(m')}\Vert_H^2 -\frac 1 2 \Vert w^{(m)}\Vert_H^2.
\end{multline}
Letting $m = 0$ in the previous equation yields
\[
 \frac 1 2 \Vert w^{(m')}\Vert_H^2  \le   \Vert \widehat{\partial} w \Vert_{V'}\Vert w\Vert_{V}+ \frac 1 2 \Vert w(0)\Vert_H^2\hbox{ for }m'=0,\ldots,N.
\]
Using the Young inequality and $\Vert \frac 1 2(a+b)\Vert \le \max(\Vert a\Vert,\Vert b\Vert)$,  and recalling that  $w(t)=\frac{1}{2} (w^{(m')} + w^{(m'+1)})$ for some $m'\in \{ 0,\cdots, N\}$, we find 
\begin{equation}\label{eq:estimusupbynorminZthetademi} 
\sup_{t\in[0,T]}\Vert w(t)\Vert_H  \le  \Big( \Vert \widehat{\partial} w \Vert_{V'}^2 + \Vert  w\Vert_V^2 +\Vert w(0)\Vert_H^2 \Big)^{1/2}\mbox{ for all }w\in W_n.
\end{equation}

From \eqref{eq:ineqestimthetademi}, we get, letting $m= 0$ and $m'=N$,
\begin{equation}\label{eq:ineqestimthetabdemi}
 \frac 1 2 (\Vert w(T)\Vert_H^2  -\Vert w(0)\Vert_H^2)  = \langle \widehat{\partial} w, w\rangle_{V',V}.
\end{equation}

{\bf Case $\Vert\Phi\Vert < 1$.} 

We deduce from \eqref{eq:ineqestimthetabdemi} that

\[
  \langle \widehat{\partial}w,w\rangle_{V',V}+  \frac  {\alpha^2 }{12\, M^3}(\Vert w \Vert_{V}^2+\Vert \widehat{\partial}w \Vert_{V'}^2)
  \ge \frac 1 2\Vert w(T)\Vert_H^2 - \frac 1 2\Vert w(0)\Vert_H^2,
  \hbox{ for all } w\in W_n.
\]
  We then obtain \eqref{eq:condlim} with
$\omega = \frac 1 2$ and $\delta = \frac 1 2 (1 - \Vert\Phi\Vert^2) $ (note that $\omega$ and $\delta$ only depend on $\Vert\Phi\Vert$).
We can now apply Lemma  \ref{lem:suffbnb}, letting $\widehat{X} = \{(\widehat{\partial} w,w,w(0),w(T)): w\in W_n\}$, since the quantities $\widehat{X}_i$ defined in the lemma satisfy $\widehat{X}_2 = V_n$ and $\widehat{X}_3 = U_n$. We get the existence of $\widehat{\beta}$,  only depending on $\alpha$, $M$ and  $\Vert\Phi\Vert$, such that
\[
	 \sup_{\substack{y\in Y_n\\ \Vert y\Vert_Y = 1}} b((\widehat{\partial} w,w),y) \ge \widehat{\beta} \Big( \Vert \widehat{\partial} w \Vert_{V'}^2 + \Vert  w\Vert_V^2 +\Vert w(0)\Vert_H^2 + \Vert w(T)\Vert_H^2\Big)^{1/2}\hbox{ for all }w\in W_n.
\]
Using \eqref{eq:estimusupbynorminZthetademi} we obtain   $\beta>0$ only depending on   $T$, $\alpha$, $M$ and  $\Vert\Phi\Vert$ such that \eqref{eq:bnbeulerimthetademi} holds.

{\bf Case $\Vert\Phi\Vert =1$.}

Letting  $m'=N$ in \eqref{eq:ineqestimthetademi}, we obtain
\[
 \frac k 2 \Vert w(T)\Vert_H^2  \le   k \Vert \widehat{\partial} w \Vert_{V'}\Vert w\Vert_{V}+ \frac k 2 \Vert w^{(m)}\Vert_H^2\hbox{ for }m=0,\ldots,N.
\]
Summing the preceding equation on $m=1,\ldots,N$ and multiplying by 2, we get
\[
  T \Vert w(T)\Vert_H^2  \le T \Vert \widehat{\partial} w \Vert_{V'}^2 + T\Vert w\Vert_{V}^2 + k\sum_{m=1}^{N}\Vert w^{(m)}\Vert_H^2.
\]
We now remark that, using $w^{(m+1)} = \frac 1 2 (w^{(m+1)}+w^{(m)}) + \frac 1 2 (w^{(m+1)}-w^{(m)})$,
\[
 k\sum_{m=1}^{N}\Vert w^{(m)}\Vert_H^2 \le 2\Vert w\Vert_{\mathbb{H}}^2 + \frac k 2\sum_{m=0}^{N-1} \Vert w^{(m+1)}-w^{(m)}\Vert_H^2.
\]
We obtain,  applying \eqref{eq:normhnormVp} in Lemma \ref{lem:normhnormVp},
\[
 k\sum_{m=1}^{N}\Vert w^{(m)}\Vert_H^2 \le 2\Vert w\Vert_{\mathbb{H}}^2 + \frac 1 2k^2 \mu_n^2 \Vert\widehat{\partial} w\Vert_{V'}^2.
\]
We thus obtain
\[
  T \Vert w(T)\Vert_H^2  \le (T +\frac 1 2k^2 \mu_n^2 )\Vert \widehat{\partial} w \Vert_{V'}^2 + T\Vert w\Vert_{V}^2.
\]
Using $k\mu_n \le C$, we get
\[
  T \Vert w(T)\Vert_H^2  \le (T +\frac 1 2C^2 )\Vert \widehat{\partial} w \Vert_{V'}^2 + T\Vert w\Vert_{V}^2.
\]
This leads to the existence of $\ctel{cst:ineqthetademi}>0$, only depending on $T$,  $\alpha$, $C$ and $M$ such that

\[
 \frac  {\alpha^2 }{12\, M^3}(\Vert \widehat{\partial} w \Vert_{V'}^2 + \Vert w\Vert_{V}^2)\ge \cter{cst:ineqthetademi}\Vert w(T)\Vert_H^2.
\]
In addition to \eqref{eq:ineqestimthetabdemi}, we obtain
\[
  \langle \widehat{\partial}w,w\rangle_{V',V}+  \frac  {\alpha^2 }{12\, M^3}(\Vert w \Vert_{V}^2+\Vert \widehat{\partial}w \Vert_{V'}^2)
  \ge (\frac 1 2+ \cter{cst:ineqthetademi})\Vert w(T)\Vert_H^2 - \frac 1 2\Vert w(0)\Vert_H^2,
  \hbox{ for all } w\in W_n.
\]
  We then obtain \eqref{eq:condlim} with
$\omega = \frac 1 2+ \cter{cst:ineqthetademi}$ and $\delta = \cter{cst:ineqthetademi} $ (note that $\omega$ and $\delta$ only depend on $\alpha$, $M$, $T$, $C$ and $C_H$).
We can now apply Lemma  \ref{lem:suffbnb}, letting $\widehat{X} = \{(\widehat{\partial} w,w,w(0),w(T)): w\in W_n\}$, since the quantities $\widehat{X}_i$ defined in the lemma satisfy $\widehat{X}_2 = V_n$ and $\widehat{X}_3 = U_n$. We get the existence of $\widehat{\beta}$,  only depending on $\alpha$, $M$ and $T$  such that
\[
	 \sup_{\substack{y\in Y_n\\ \Vert y\Vert_Y = 1}} b((\widehat{\partial} w,w),y) \ge \widehat{\beta} \Big( \Vert \widehat{\partial} w \Vert_{V'}^2 + \Vert  w\Vert_V^2 +\Vert w(0)\Vert_H^2 + \Vert w(T)\Vert_H^2\Big)^{1/2}\hbox{ for all }w\in W_n.
\]
Using \eqref{eq:estimusupbynorminZthetademi} we obtain  $\beta$ (with the same dependencies) such that \eqref{eq:bnbeulerimthetademi} holds.  

\end{proof}
\begin{rem}\label{rem:4.5}
	
	 We cannot expect a BNB condition if $\theta=\frac 1 2$ and $\Vert\Phi\Vert = 1$ without further assumptions. In fact, if $\theta = \frac 1 2$ and $\Vert\Phi\Vert=1$, a BNB condition implies a CFL condition. Indeed, let $\Phi = {\rm Id}$, $T = 1$, $N\neq 0$ an even integer. Let $u\in U_n$ and let $w^{(m)} = (-1)^{m} k u$. Define $w\in W_n$ by \eqref{eq:schememtheta}. Since $\Phi w(T) = w(0)$ and $w(t)=0$ for a.e. $t$, we seek $(y,z)\in Y_n$ with $\Vert(y,z)\Vert_Y = 1$ providing the maximum value for $b((\widehat{\partial}w,w),(y,z)$. We have 
 \[
  b((\widehat{\partial}w,w),(y,z)) = \sum_{m=0}^{N-1} k \langle (-1)^{m+1}2 u,y^{(m)}\rangle_H.
 \]
 The maximum value is obtained for $y^{(m)} = (-1)^{m+1} \alpha u$ and  $z = 0$ where $\alpha = 1/\Vert u\Vert_U$. Then
 \[
  b((\widehat{\partial}w,w),(y,z)) = \sum_{m=0}^{N-1} k \langle (-1)^{m+1} 2 u,y^{(m)}\rangle_H = 2 \frac {\Vert u\Vert_H^2} {\Vert u\Vert_U}.
 \]
 The inequality
 \[
  \beta \sup_{t\in[0,T]} \Vert w(t)\Vert_H = \beta k \Vert u\Vert_H\le 2 \frac {\Vert u\Vert_H^2} {\Vert u\Vert_U}
 \]
implies
 \[
  \beta k \|u\|_H\le 2  \frac{\|u\|^2_H}{\Vert u\Vert_U}.
 \]
 Then we deduce that 
 \[  k\mu_n \leq \frac{2}{\beta}.\]
 Thus, the condition considered in Lemma~\ref{lem:estimnoruTHthetademi} in the case $\|\Phi\|=1$ is fulfilled for $C=\frac{2}{\beta}$.
\end{rem}

\begin{lem}\label{lem:estimnoruTHthetaim}  
 
Let $\theta\in (\frac 1 2,1]$. There exists a constant $\beta>0$, only depending on $\alpha$, $M$, $T$, $\theta$ and $C_H$, such that 
	\begin{equation}\label{eq:bnbeulerimtheta}
	 \sup_{\substack{y\in Y_n\\ \Vert y\Vert_Y = 1}} b(x,y) \ge \beta \Vert x\Vert_Z \hbox{ for all }x\in X_n.
	\end{equation}
 \end{lem}
\begin{proof} 
We have, for $0\le m < m'\le N$,
\begin{multline}\label{eq:ineqestimthetaim}
 \int_{mk}^{m'k} \langle \partial w(t), w(t)\rangle_H{\rm d}t = \sum_{p=m}^{m'-1} k \langle \frac {w^{(p+1)} - w^{(p)}} k,\theta w^{(p+1)}+(1-\theta)w^{(p)}\rangle_H \\
 = \frac 1 2 \Vert w^{(m')}\Vert_H^2 +(\theta-\frac 1 2)\sum_{p=m}^{m'-1} \Vert w^{(p+1)} - w^{(p)}\Vert_H^2 -\frac 1 2 \Vert w^{(m)}\Vert_H^2.
\end{multline}
Letting $m = 0$ in the previous equation provides, since $\theta\ge\frac 1 2$,
\[
 \frac 1 2 \Vert w^{(m')}\Vert_H^2  \le   \Vert \widehat{\partial} w \Vert_{V'}\Vert w\Vert_{V}+ \frac 1 2 \Vert w(0)\Vert_H^2\hbox{ for }m'=0,\ldots,N.
\]
Using the Young and the triangle inequalities, this gives 
\begin{equation}\label{eq:estimusupbynorminZthetaim} 
\sup_{t\in[0,T]}\Vert w(t)\Vert_H  \le \Big( \Vert \widehat{\partial} w \Vert_{V'}^2 + \Vert  w\Vert_V^2 +\Vert w(0)\Vert_H^2 \Big)^{1/2}\mbox{ for all }w\in W_n.
\end{equation}

From \eqref{eq:ineqestimthetaim}, we get, letting $m= 0$ and $m'=N$,
\begin{equation}\label{eq:ineqestimthetabim}
 \frac 1 2 (\Vert w(T)\Vert_H^2  -\Vert w(0)\Vert_H^2)  \le  \langle \widehat{\partial} w, w\rangle_{V',V}.
\end{equation}

We notice that, in the case $\Vert\Phi\Vert < 1$ we could conclude as in the proof of Lemma \ref{lem:estimnoruTHthetademi}. But our aim is to handle the general case $\Vert\Phi\Vert \le 1$.

Letting  $m'=N$ in \eqref{eq:ineqestimthetaim}, we obtain
\[
 \frac k 2 \Vert w(T)\Vert_H^2  \le   k \Vert \widehat{\partial} w \Vert_{V'}\Vert w\Vert_{V}+ \frac k 2 \Vert w^{(m)}\Vert_H^2\hbox{ for }m=0,\ldots,N.
\]
Summing the preceding equation on $m=1,\ldots,N$ and multiplying by 2, we get
\[
  T \Vert w(T)\Vert_H^2  \le 2T \Vert \widehat{\partial} w \Vert_{V'}\Vert w\Vert_{V} + k\sum_{m=1}^{N}\Vert w^{(m)}\Vert_H^2.
\]
We again remark that
\begin{multline*}
 \Vert w\Vert_{\mathbb{H}}^2 = \sum_{m=0}^{N-1}k \Vert\theta w^{(m+1)}+(1-\theta)w^{(m)}\Vert_H^2\\
 \ge (2\theta-1)^2\sum_{m=1}^{N}k \Vert w^{(m)}\Vert_H^2  + (1-\theta)(2\theta-1)k(\Vert w(T)\Vert_H^2- \Vert w(0)\Vert_H^2).
\end{multline*}
The two preceding relations give
\[
 \frac {C_H^2}{(2\theta-1)^2}\Vert w\Vert_{V}^2 
    +  T (\Vert \widehat{\partial} w \Vert_{V'}^2 + \Vert w\Vert_{V}^2)\ge (T+\frac {1-\theta}{2\theta-1}k) \Vert w(T)\Vert_H^2 - \frac {1-\theta}{2\theta-1}k\Vert w(0)\Vert_H^2.
\]
This leads to the existence of $\ctel{cst:ineqthetaim}>0$, only depending on $T$, $C_H$, $\alpha$ and $M$ such that

\[
 \frac  {\alpha^2 }{12\, M^3}(\Vert \widehat{\partial} w \Vert_{V'}^2 + \Vert w\Vert_{V}^2)\ge \cter{cst:ineqthetaim}\Big( (1+\frac {1-\theta}{2\theta-1}\frac k T)\Vert w(T)\Vert_H^2 -\frac {1-\theta}{2\theta-1}\frac k T\Vert w(0)\Vert_H^2\Big).
\]
In addition to \eqref{eq:ineqestimthetab}, we obtain
\[
  \langle \widehat{\partial}w,w\rangle_{V',V}+  \frac  {\alpha^2 }{12\, M^3}(\Vert w \Vert_{V}^2+\Vert \widehat{\partial}w \Vert_{V'}^2)
  \ge \mu \Vert w(T)\Vert_H^2 - \nu\Vert w(0)\Vert_H^2,
  \hbox{ for all } w\in W_n,
\]
with
\[
 \mu = \frac {1} 2 + \cter{cst:ineqthetaim} (1+\frac {1-\theta}{2\theta-1}\frac k T)\hbox{ and }\nu = \frac {1} 2 + \cter{cst:ineqthetaim}\frac {1-\theta}{2\theta-1}\frac k T.
\]

  We then obtain \eqref{eq:condlim} with
$\omega = \frac 1 2+ \cter{cst:ineqthetaim} (1+\frac {1-\theta}{2\theta-1})$ and $\delta = \cter{cst:ineqthetaim} $ (note that $\omega$ and $\delta$ only depend on $\alpha$, $M$, $T$, $\theta$ and $C_H$).
We can now apply Lemma  \ref{lem:suffbnb}, letting $\widehat{X} = \{(\widehat{\partial} w,w,w(0),w(T)): w\in W_n\}$, since the quantities $\widehat{X}_i$ defined in the lemma satisfy $Y_n = \widehat{X}_2\times \widehat{X}_3$. We get the existence of $\widehat{\beta}$,  only depending on $\alpha$, $M$, $T$, $\theta$ and $C_H$, such that
\[
	 \sup_{\substack{y\in Y_n\\ \Vert y\Vert_Y = 1}} b((\widehat{\partial} w,w),y) \ge \widehat{\beta} \Big( \Vert \widehat{\partial} w \Vert_{V'}^2 + \Vert  w\Vert_V^2 +\Vert w(0)\Vert_H^2 + \Vert w(T)\Vert_H^2\Big)^{1/2}\hbox{ for all }w\in W_n.
\]
Using \eqref{eq:estimusupbynorminZthetaim} we obtain the constant $\beta$  (with the same dependencies) such that \eqref{eq:bnbeulerimtheta} holds. 

\end{proof}

\begin{rem}~
	\begin{enumerate}
		\item[(a)]   Note that $\beta$ does not depend on $N$ or on $U_n$. This will be important when we discuss convergence.
		\item[(b)] Of course in the three previous  lemmas establishing the BNB-condition, there exists an optimal $\beta>0$. Since the continuity constant of $b$ is $2M$, an upper limit of $\beta$ is $2M$.     
	\end{enumerate}

\end{rem}

\subsection{Conclusion: optimal error estimate}\label{sub:4.2err}

 We now obtain  the following error estimate for a solution $u$ of Problem~\eqref{eq:5.1}. Only later we will deduce from our estimates that a unique solution of \eqref{eq:5.1} exists. 
\begin{thm}\label{thm:erresttheta} Let $\theta\in [0,+\infty)$ be given.  Let $U_n\subset U$ be a finite dimensional subspace of $U$ and let $N\in\mathbb{N}$.
In the case where $\theta\in[0,\frac 1 2)$ assume that $N$ is so large that the CFL-condition \eqref{eq:cfltheta} is satisfied. 

Then for each $f\in V'$ and $\xi_0\in H$, the $\theta-$scheme \eqref{eq:schememextheta} has a unique solution to which we associate the step-function $w_n\in W_n$ via \eqref{eq:schememtheta}. Moreover, there exists $\beta>0$ such that
	\begin{multline}\label{eq:erresttheta}
		\Big(\Vert u' - \widehat{\partial}w_n\Vert_{V'}^2 +  \Vert u - w_n\Vert_{V}^2 + \sup_{t\in[0,T]} \Vert u(t) - w_n(t)\Vert_{H}^2\Big)^{1/2} \le \\
(1 + \frac {2M}{\beta})  \inf\{\Big(\Vert u' - \widehat{\partial}w\Vert_{V'}^2 +  \Vert u - w\Vert_{V}^2 + \sup_{t\in[0,T]} \Vert u(t) - w(t)\Vert_{H}^2\Big)^{1/2}: w\in W_n\},
	\end{multline}
where $u$ is a solution of \eqref{eq:5.1}.
The constant $\beta>0$ may be chosen depending only on
\begin{itemize}
		\item [a)] $\theta$, $M$, $\alpha$, $T$ and $C_H$ if $0\le \theta<\frac 1 2$,
		\item [b)] $M$, $\alpha$, $T$ and $C$ if $\theta = \frac 1 2$ and $\Vert\Phi\Vert = 1$ where $C>0$ is a constant such that $k\mu_n\le C$,
		\item [c)] $M$, $\alpha$, $T$ and $\Vert\Phi\Vert$ if $\theta = \frac 1 2$ and $\Vert\Phi\Vert < 1$,
		\item [d)] $\theta$, $M$, $\alpha$, $T$ and $C_H$ if $\frac 1 2 < \theta\le 1$.
	\end{itemize}
	
	\end{thm}

 \begin{proof}
 Recall the definitions of $Z$ and $Y$ given in Section \ref{sec:extended}. Choose $\beta>0$, as given by Lemma \ref{lem:estimnoruTHthetaex} in the case $\theta\in[0,\frac 1 2)$, by Lemma \ref{lem:estimnoruTHthetademi} in the case $\theta =\frac 1 2$ and by Lemma \ref{lem:estimnoruTHthetaim} for $\theta\ge \frac 1 2$.  Thus $b$ satisfies the BNB-condition \eqref{eq:betabnbn}. Since ${\rm dim}X_n ={\rm dim}Y_n <\infty$, Problem \eqref{eq:schemeinZY} has a unique solution $w_n\in W_n$, where $w_n$ is given via \eqref{eq:schememtheta} from $(w^{(m)})_{m=0,\ldots,N}\subset U_n$ satisfying \eqref{eq:schemeiniex} and \eqref{eq:schememextheta}, see Lemma \ref{lem:schemeexpabs}. Recall from \eqref{eq:defb} that the continuity constant of $b$ is $2M$.

 Recalling that
 \[
  \vert (z_1,z_2)\Vert_Z = \Big(\Vert z_1\Vert_{V'}^2 +  \Vert z_2\Vert_{V}^2 +  \sup_{t\in[0,T]} \Vert z_2(t)\Vert_{H}^2\Big)^{1/2},
 \]
 we see that Lemma \ref{lem:errest} implies \eqref{eq:erresttheta}.
 \end{proof}

\subsection{Interpolation results}\label{sub:4.3}

 In this subsection we estimate the distance in $Z$ of a given function $w\in W$ from $W_n$ (recall that $W_n$ is not a subspace of $W$).

 Let $\theta\in [0,1]$ be given. 
  
 For any $w\in H^1(0,T;U)$, we define an interpolation of $w$ in $W_n$ by 
 \begin{equation}\label{eq:definterpex}
 w_n\in W_n\mbox{ defined by \eqref{eq:schememtheta} using the sequence }  (P^U_n w(mk))_{m=0,\ldots,N},
 \end{equation}
 where we recall that we denote by $P^U_n~:~U\to U_n$ the orthogonal projection on $U_n$ in $U$.
Let $\delta_n: V\to [0,+\infty)$ be defined by
 \begin{equation}\label{eq:definterVex}
\delta_n(v) :=  \Vert v - P^U_n\circ v\Vert_V = \left(\int_0^T \Vert v(t) - P^U_n v(t)\Vert_U^2{\rm d}t\right)^{1/2}\mbox{ for all }v\in V.
 \end{equation}
 
 We start by a simple lemma whose proof can be omitted.
 
 \begin{lem}\label{lem:rieszUV}
  Let $R^U:U'\to U$ (resp. $R^V:V'\to V$) be the Riesz isomorphism, given by $\langle r,u\rangle_{U',U} = \langle R^U r,u\rangle_{U}$ for all $(r,u)\in U'\times U$ (resp.  $\langle f,v\rangle_{V',V} = \langle R^V f,v\rangle_{V}$ for all $(f,v)\in V'\times V$). Then 
  \[
   \langle R^V f,v\rangle_{V} = \int_0^T \langle R^U f(t),v(t)\rangle_{U}{\rm d}t.
  \]
  Using the Gelfand triple embeddings, we have
  \[
    \langle R^V w,v\rangle_{V} =  \langle w,v\rangle_{\mathbb{H}} \hbox{ for all }v,w\in V.
  \]
  In the case where $w\in H^1(0,T;U)$, we have $R^V w\in H^1(0,T;U)$ with $(R^V w)' = R^V( w')$.
 \end{lem}

 We then have the following lemma.

 \begin{lem}\label{lem:interpinVex} 
 Let $w\in H^1(0,T;U)$ and let $w_n\in W_n$ be defined by \eqref{eq:definterpex}.
 Then 
 \[
   \Vert w - P^V_n w \Vert_V\le \Vert w - w_n\Vert_{V} \le \delta_n(w) +  k  \Vert w'\Vert_{V}
 \]
 \end{lem}
 \begin{proof}
We can write
\[
 \Vert w - w_n\Vert_{V}\le \Vert w - P^U_n\circ w\Vert_{V}+\Vert P^U_n\circ  w - w_n\Vert_{V}.
\]
We have, for all $m=0,\ldots,N-1$ and for a.e. $t\in(mk,(m+1)k)$,
\begin{multline*}
 \Vert P^U_n w(t) - w_n(t)\Vert_{U} = \Vert P^U_n w(t) - (1-\theta) P^U_n w(mk) -\theta P^U_n w((m+1)k)\Vert_{U}\\
 \le (1-\theta)\Vert w(t) - w(mk)\Vert_{U}+\theta\Vert w(t) - w((m+1)k)\Vert_{U} \\
 = (1-\theta)\Vert\int_{mk}^{t} w'(s){\rm d}s\Vert_U +\theta\Vert\int_{t}^{(m+1)k} w'(s){\rm d}s\Vert_U\le \int_{mk}^{(m+1)k} \Vert w'(s)\Vert_U{\rm d}s.
\end{multline*}
Note that $P^U_n w(t) - w_n(t)=0$ for $t = mk$ and $t = (m+1)k$.
This yields, owing to the Cauchy-Schwarz inequality,
\[
 \Vert P^U_n w(t) - w_n(t)\Vert_{U}^2 \le k \int_{mk}^{(m+1)k} \Vert w'(s)\Vert_U^2{\rm d}s,
\]
and therefore
\[
  \int_{mk}^{(m+1)k}\Vert P^U_n w(t) - w_n(t)\Vert_{U}^2{\rm d}t  \le k^2 \int_{mk}^{(m+1)k} \Vert w'(s)\Vert_U^2{\rm d}s,
\]
which concludes the proof of the lemma.
 \end{proof}

 \begin{lem}\label{lem:cvschemeex}  
Let $w\in H^2(0,T;U)$ and let $w_n\in W_n$ be defined by \eqref{eq:definterpex}.
 Then
 \begin{equation}\label{eq:cvschemeex}
  \Vert w' - \widehat{\partial}w_n\Vert_{V'}  \le \delta_n(R^V w')+ k \Vert R^V w''\Vert_V+ C_H^2 (\delta_n(w') + k\Vert w''\Vert_V).
 \end{equation}
 \end{lem}
 \begin{proof}
 Let $\widehat{w_n'}\in V'$ be defined by
 \begin{equation}\label{eq:defrhhex}
  \langle \widehat{w_n'}, v\rangle_{V',V} = \langle w', P^V_n v\rangle_{V',V} = \langle w', P^V_n v\rangle_{\mathbb{H}} \mbox{ for all }v\in V.
 \end{equation}
 We notice that
 \[
  \Vert w' - \widehat{\partial}w_n\Vert_{V'} \le \Vert w' - \widehat{w_n'}\Vert_{V'} + \Vert \widehat{w_n'} - \widehat{\partial}w_n\Vert_{V'}.
 \]
 We have, for all $v\in V$,
 \begin{multline*}
 \langle w' - \widehat{w_n'}, v\rangle_{V',V} = \langle w', v\rangle_{V',V} -\langle w' , P^V_n v\rangle_{V',V} \\ 
 =  \langle R^V w' , v\rangle_{V} -\langle R^V w' , P^V_n v\rangle_{V} = \langle R^V w' - P^V_n R^V w' ,v\rangle_{V}.
\end{multline*}
 This proves that
 \[
  \Vert  w' - \widehat{w_n'}\Vert_{V'} = \Vert R^V w' - P^V_n R^V w'\Vert_{V}.
 \]
 Using Lemma \ref{lem:interpinVex} with $R^V w'$ instead of $w$, we get the first two terms of the right-hand-side of \eqref{eq:cvschemeex}.
 We also have that
  \begin{multline*}
  \Vert  \widehat{w_n'} - \widehat{\partial}w_n\Vert_{V'} = \sup_{v\in V,\Vert v\Vert_V = 1} \langle \widehat{w_n'} - \widehat{\partial} w_n, v\rangle_{V',V}\\
  =  \sup_{v\in V,\Vert v\Vert_V = 1} \langle w' - \partial w_n,P^V_n v\rangle_{\mathbb{H}}\le C_H \Vert w' - \partial w_n\Vert_{\mathbb{H}}.
\end{multline*}
 Using \eqref{eq:projvproju}, we observe that, for a.e. $t\in(mk,(m+1)k)$ and all $m=0,\ldots,N-1$,
 \begin{multline*}
 P^V_n (w')(t) =  \frac 1 k\int_{mk}^{(m+1)k} P^U_n (w')(s){\rm d}s \\
 =  \frac 1 k\int_{mk}^{(m+1)k} (P^U_n w)'(s){\rm d}s =\frac {P^U_n w((m+1)k) - P^U_n w(m k)} k = \partial w_n(t).
 \end{multline*}
 Therefore, $P^V_n (w') = \partial w_n$, and we get
\[
  \Vert w' - \partial w_n\Vert_{\mathbb{H}}  =  \Vert w' - P^V_n (w')\Vert_{\mathbb{H}} \le C_H \Vert w' - P^V_n (w')\Vert_{V}.
 \]
 We then use Lemma \ref{lem:interpinVex} with $w'$ instead of $w$ to conclude the proof.
  \end{proof}

 \begin{lem}\label{lem:interpinVpex}  Let $w\in H^2(0,T;U)$ and let $w_n\in W_n$ be defined by \eqref{eq:definterpex}. Recalling that $C_T$ is introduced in \eqref{eq:embWinH}, we have
 \[
 \sup_{t\in[0,T]}\Vert w_n(t) - w(t)\Vert_H \le k C_H (\frac 1 {\sqrt{T}} \Vert w'\Vert_V+ \sqrt{T}\Vert w''\Vert_V)+ C_T(\delta_n(w) + C_H^2\delta_n(w')).
 \]
 \end{lem}
 \begin{proof}
 Let us first compute a preliminary inequality.
 We observe that, for $s,t\in[0,T]$,
\[
 w'(t) = w'(s) + \int_s^t w''(\tau){\rm d}\tau,
\]
which leads to 
\[
\Vert w'(t)\Vert_U \le \Vert w'(s)\Vert_U + \int_0^T  \Vert w''(\tau)\Vert_U{\rm d}\tau.
\]
Integrating with respect to $s$ and using the Cauchy-Schwarz inequality, we obtain
\begin{equation}\label{eq:poincmoy}
 \sup_{t\in[0,T]} \Vert w'(t)\Vert_U\le \frac 1 {\sqrt{T}} \Vert w'\Vert_V+ \sqrt{T}\Vert w''\Vert_V.
\end{equation}

 For all $t\in[0,T]$, we have
 \begin{multline*}
  \Vert w_n(t) - w(t)\Vert_H \le \Vert w_n(t) - P^U_n w(t)\Vert_H + \Vert P^U_n w(t) - w(t)\Vert_H \\
  \le C_H \Vert w_n(t) - P^U_n w(t)\Vert_U + \Vert P^U_n w(t) - w(t)\Vert_H.
 \end{multline*}
 In the proof of Lemma \ref{lem:interpinVex}, we show that for all $m=0,\ldots,N-1$ and all $t\in[mk,(m+1)k]$, we have
\[
\Vert P^U_n w(t) - w_n(t)\Vert_{U} \le  \int_{mk}^{(m+1)k} \Vert w'(s)\Vert_U{\rm d}s.
\]
Using \eqref{eq:poincmoy}, we obtain
\[
 \Vert P^U_n w(t) - w_n(t)\Vert_{U} \le k(\frac 1 {\sqrt{T}} \Vert w'\Vert_V+ \sqrt{T}\Vert w''\Vert_V).
\]

Now, we write, owing to \eqref{eq:embWinH},
\[
 \Vert P^U_n w(t) - w(t)\Vert_H \le C_T(\Vert P^U_n\circ  w - w\Vert_V^2 + \Vert P^U_n\circ  w' - w'\Vert_{V'}^2)^{1/2}\le C_T(\Vert P^U_n\circ  w - w\Vert_V + \Vert P^U_n\circ  w' - w'\Vert_{V'}).
\]
We now write
\begin{multline*}
  \Vert  P^U_n\circ  w' - w'\Vert_{V'} = \sup_{v\in V,\Vert v\Vert_V = 1} \langle P^U_n\circ  w' - w', v\rangle_{V',V}\\
  =  \sup_{v\in V,\Vert v\Vert_V = 1} \langle P^U_n\circ  w' - w', v\rangle_{\mathbb{H}}\le C_H \Vert P^U_n\circ  w' - w'\Vert_{\mathbb{H}}.
\end{multline*}
We then get that
\[
 \Vert P^U_n\circ  w' - w'\Vert_{\mathbb{H}}\le C_H \Vert P^U_n\circ  w' - w'\Vert_{V} = C_H \delta_n(w').
\]
This concludes the proof of the lemma.  
 \end{proof}

 \subsection{The convergence result}
 
 Using the results of Section \ref{sub:4.3}, we first show that we can approximate each function of $W$ by elements of $W_n$. For this we now consider a sequence of finite-dimensional subspaces $(U_n)_{n\in\mathbb{N}}$ of $U$ such that
 \[
  \displaystyle\lim_{n\to\infty} U_n = U.
 \]
This is equivalent to saying
 \[
  \displaystyle\lim_{n\to\infty} \Vert  u -  P^U_n u\Vert_U =0\hbox{ for all }u\in W,
 \]
 where $P^U_n$ denotes the orthogonal projection of $U$ to $U_n$.
 
 For each $n\in\N$ we let $\mu_n$ be the best constant such that $\|u\|_U\leq \mu_n\|u\|_H$ for all $u\in U_n$.  
 
 Let $N_n\in \mathbb{N}$ such that $\displaystyle\lim_{n\to \infty} N_n = \infty$. For each $n\in\mathbb{N}$, let $k_n = \frac {T}{N_n}$ be the time step and let
 \[
  V_n=\{v:(0,T)\to U_n: v\hbox{ is constant on }(mk_n,(m+1)k_n)\hbox{ for }m=0,\ldots,N_n - 1\}.
 \]
 The defintion of $W_n$ is given in \eqref{eq:schememtheta}. 
 
 Recall that $W = H^1(0,T;U')\cap V $ and that the spaces $Z$, $Y$, $X$ and $X_n$ are defined by
\[
 Z = V'\times G\hbox{ and } Y = V\times H, 
\]
\[
 X = \{(w',w): w\in W\},\  X_n = \{(\widehat{\partial}w,w): w\in W_n\}\hbox{ and }\]
 \[Y_n =  V_n\times H.
  \]
Thus, $X_n$ and $X$ are subspaces of $Z$ and $Y_n$ is a subspace of $Y$. Note that
\[
 X_n\cap X = \{ (0,\xi):\xi\in U_n\}
\]
 is a very small space. Nonetheless we will show that, under the CFL-condition, one has $\displaystyle\lim_{n\to\infty} X_n=X$ . We first prove one inclusion without any further condition. 
 \begin{thm}\label{thm:3.9}
 Let $(N_n)_{n\in\mathbb{N}}$ be a sequence of elements of $\mathbb{N}$ which tends to infinity.
   Let $w\in W$. Then there exist $w_n\in W_n$ for all $n\in\mathbb{N}$ such that
  \begin{equation}\label{eq:4.22}
   \Vert w' - \widehat{\partial}w_n\Vert_{V'}^2 +  \Vert w - w_n\Vert_{V}^2 + \sup_{t\in[0,T]} \Vert w(t) - w_n(t)\Vert_{H}^2 \to 0 \hbox{ as }n\to\infty. 
  \end{equation}

This implies in particular that  
$\displaystyle X\subset \displaystyle\lim_{n\to\infty} X_n$.
 
\end{thm}
 \begin{proof}
 We use the preceding results, letting $\theta=0$.
 Let us first observe that the property
\[
  \displaystyle\lim_{h\to 0} \Vert u - P^U_n u\Vert_U = 0\mbox{ for all }u\in U 
\]
implies, by Lebesgue's convergence theorem,  that
\[
 \displaystyle\lim_{h\to 0} \delta_n(v)  = \displaystyle\lim_{h\to 0} \Vert v - P^U_n\circ  v\Vert_V = 0\mbox{ for all }v\in V.
\]
 
  Let $w\in W$ and let $\varepsilon>0$. By the proof of \cite[III. Proposition 1.2, p. 106]{Sho97} there exists $v\in H^2(0,T;U)$ such that
  \[
    \Vert w' - v'\Vert_{V'} +  \Vert w - v\Vert_{V}\le \varepsilon.
  \]
  By \eqref{eq:embWinH} this implies that $\sup_{t\in[0,T]}\Vert w(t) - v(t)\Vert_{H}\le C_H \varepsilon$.
  Let $v_n\in W_n$ be given from $v$ via \eqref{eq:definterpex}, which means that $v_n(t) = P^U_n v(m k_n)$ on $[m k_n,(m+1)k_n)$ for $m=0,\ldots,N_n-1$ and $v_n(T) = P^U_n v(T)$. Then, by Lemma \ref{lem:interpinVex},  $\Vert v_n - v\Vert_V\to 0$ as $n\to\infty$. By Lemma \ref{lem:cvschemeex}, $\displaystyle\lim_{n\to\infty}\Vert v' - \widehat{\partial}v_n\Vert_{V'} = 0$ and, by Lemma \ref{lem:interpinVpex},  $\displaystyle\lim_{n\to\infty}\sup_{t\in[0,T]}\Vert v(t) - v_n(t)\Vert_{H} = 0$.

  Since $\varepsilon>0$ is arbitrary, the proof is complete.
  
  The preceding proof shows that $\displaystyle X\subset \displaystyle\lim_{n\to\infty} X_n$.
 \end{proof}
Next we show that the discrete solutions converge to the solution of the continuous problem. 
We know that a unique solution of Problem~\eqref{eq:5.1} exists (see Theorem~\ref{thm:5.3} and the comments after it). But our estimates give also a  new proof of this result which we incorporate into the formulation and the proof of the following  theorem.  Recall that $\Phi\in {\mathcal{L}}(H)$ and $\Vert\Phi\Vert\le 1$.

 \begin{thm}\label{thm:cv} Let $f\in V'$ and $\xi_0\in H$.  Then Problem~\eqref{eq:5.1} has a unique solution $u\in W$. 

Let $(N_n)_{n\in\mathbb{N}}$ be a sequence of elements of $\mathbb{N}$ which tends to infinity such that the following holds.
	\begin{itemize}
		\item []{\bf Case  $\theta\in[0,\frac 1 2)$.} We assume that Condition \eqref{eq:cfltheta} is fulfilled. 
		\item []{\bf Case  $\theta = \frac 1 2$ and $\Vert\Phi\Vert = 1$.} We assume that $\sup_{n\in\mathbb{N}} \frac {\mu_n} {N_n} <\infty$.
		\item []{\bf Case  $\theta = \frac 1 2$ and $\Vert\Phi\Vert < 1$.} No additional condition is required. 
		\item []{\bf Case  $\theta \in (\frac 1 2,1]$.} No additional condition is required.
	\end{itemize}

 Then the  Euler Scheme \eqref{eq:schemeiniex}-\eqref{eq:schememextheta} has a unique solution to which we associate the step function $w_n\in W_n$ via \eqref{eq:schememtheta}.  These discrete solutions $w_n$ converges to $u$ in the following sense:
 \begin{equation}\label{eq:4.23}
   \displaystyle\lim_{n\to\infty} \Big(\Vert u' - \widehat{\partial}w_n\Vert_{V'} +  \Vert u - w_n\Vert_{V} + \sup_{t\in[0,T]} \Vert u(t) - w_n(t)\Vert_{H}\Big) = 0
 \end{equation}
 \end{thm}
 \begin{proof} 
 We apply Theorem~\ref{thm:bnbnonconf}	in our situation. There are three hypotheses to be verified. Lemma~\ref{lem:estimnoruTHthetaex}, Lemma~\ref{lem:estimnoruTHthetademi} and Lemma~\ref{lem:estimnoruTHthetaim}  give the BNB-estimate \eqref{eq:bnbeulerextheta}. Lemma~\ref{lem:xhasdenserange} shows the dual uniqueness condition~\eqref{eq:denserange} for $X$.  
 Finally, Theorem~\ref{thm:3.9} shows that $X\subset \displaystyle\lim_{n\to \infty}X_n$. 
 Now Theorem~\ref{thm:bnbnonconf} asserts that  Problem \eqref{eq:pbcontinZ} has a unique solution, and this is equivalent to Problem~\eqref{eq:5.1} having a unique solution. Theorem~\ref{thm:bnbnonconf} also gives the required convergence of the Euler scheme.  This finishes the proof of the theorem.     
 	  \end{proof}
 	  
\begin{rem}\label{rem4.14}
In the proof we used  Theorem \ref{thm:bnbnonconf} which also implies that $Y = \mathop{\rm w\mbox{-}lim}\limits_{n\to\infty} Y_n$.
 But in this case, Lemma~\ref{lem:interpinVex}  also shows that $Y = \displaystyle\lim_{n\to\infty} Y_n$.
\end{rem}
\begin{rem}
	In the case $\theta=\frac{1}{2}$, $\Phi=\Id$ and $T=1$, the conclusion of Theorem~\ref{thm:cv} is false if $\displaystyle\sup_{n\in\N}\frac{\mu_n}{N_n}=\infty$.  In fact, let $L\in Y'$. Using \eqref{eq:bnbeulerimthetademi}, for each $n\in\N$ we find a unique $x_n\in X_n$ such that $b(x_n,y)=\langle L,y\rangle_{Y',Y}$ for all $y\in Y_n$.  Then, for each $n\in\N$,  there exists $w_n\in W_n$ such that $x_n=(\widehat{\partial}w_n,w_n)$. If the conclusion of Theorem ~\ref{thm:cv} is true, it follows in particular that $(x_n)_{n\in\N}$ converges in $Z$.  Then Theorem~\ref{thm:bnbnonconfrec} yields $\beta>0$ such that 
	\[    \sup_{\|y\|_Y\leq 1,y\in Y_n} b(x,y)\geq \beta \|x\|_Z\]
	for all $x\in X_n$ and all $n\in \N$.  Then Remark~\ref{rem:4.5} implies that $k_n\mu_n\leq \frac{2}{\beta}$, where k$_n=\frac{1}{N_n}$. Thus $\frac{\mu_n}{N_n}\leq \frac{2}{\beta}$ for all $n\in \N$.  
\end{rem}

Since the estimate \eqref{eq:erresttheta} holds, the speed of convergence in \eqref{eq:4.23} is optimal. Theorem \ref{thm:cv} also makes more precise the interpolation result, Theorem \ref{thm:3.9}. It says that we can approximate each $w\in W$ by discrete functions $w_n$ in the sense of \eqref{eq:4.23}. We now show  that also the converse is true: the functions in $W$ are exactly those functions which can be approximated by discrete elements in the sense of \eqref{eq:4.23}.
  In other words, the following surprising corollary holds. 
 
 \begin{cor}\label{cor:4.20} Under the assumptions of Theorem \ref{thm:cv} one has
 \[X=\displaystyle\lim_{n\to \infty} X_n.\]
 This means the following. Let 
 \begin{multline*}
  W_n = \{w:[0,T]\to U_n: w(t) = (1-\theta)w(mk)+\theta w((m+1)k)\hbox{ on }(m k_n,(m+1)k_n)\\
  \hbox{ for all }m=0,\ldots,N_n - 1\}.
 \end{multline*}

 Then for $(\widetilde{v},v)\in V'\times G$, the following assertions are equivalent.
 \begin{itemize}
  \item [(i)] There exists $w_n\in W_n$ such that
  \[
   \Vert \widetilde{v} - \widehat{\partial}w_n\Vert_{V'}^2 +  \Vert v - w_n\Vert_{G}^2 \to 0 \hbox{ as }n\to\infty;
  \]
  \item [(ii)] $v\in W$, $\widetilde{v} = v'$.
 \end{itemize} 
 \end{cor}
 \begin{proof} 
 	We have seen in Theorem~\ref{thm:3.9} that $X\subset \displaystyle\lim_{n\to \infty}X_n$. The reverse inclusion follows  from Theorem~\ref{thm:bnbnonconf}.   Moreover, Condition (i) means that $x := (\widetilde{v},v)\in \displaystyle\lim_{n\to\infty} X_n$, and   $x\in X$ exactly means (ii).
  
 \end{proof}

\section{Discontinuous Galerkin scheme}\label{sec:dg}
We focus in this section on the time discontinuous Galerkin scheme. We again keep the framework of Section \ref{sec:pbcont}. 
	\subsection{Description of the scheme}\label{sub:51}
 As in Section \ref{sec:eulerex},  let $U_n$ be a finite dimensional subspace of $U$, $P^U_n$ be the orthogonal projection on $U_n$  for the scalar product of $U$, $N\in\mathbb{N}^\star$, and define $k = \frac {T}{N}$. 
 Let $q\in \mathbb{N}$.
 For any vector space $E$ we denote by  $\mathcal{P}^q(\mathbb{R};E)$ the set of all functions $P:\mathbb{R}\to E$ such that there exists $p_0,\ldots,p_q\in E$ with
  \[
  P(t) = \sum_{i=0}^q t^i p_i\mbox{ for all }t\in \mathbb{R}.
  \]
Let $W_n\subset G$ be the set of all functions $w~:~[0,T]\to U_n$ verifying the following property:
 for all $m=1,\ldots,N$, the restriction of $w$ to $((m-1)k,mk]$ is an element of $\mathcal{P}^q(\mathbb{R};U_n)$, which means that
 there exist $q+1$ elements of $U_n$, denoted by $w_0^{(m)},\ldots,w_q^{(m)}$ such that
  \[
   w(t) = \sum_{i=0}^q \Big(\frac {t-(m-1)k} k\Big)^i w_i^{(m)}\mbox{ for all }t\in ((m-1)k,mk].
  \] 
 We then define the ``broken derivative'' $\partial_b w\in V$, for any $w\in W_n$, as the function equal to $w'(t)$ for a.e. $t\in ((m-1)k,mk)$ (we have $\partial_b w = 0$ if $q=0$).
 
 The time discontinuous Galerkin scheme  \cite{eriksson1985dg,saito2021var} amounts to finding  $w\in W_n$ such that
\begin{equation}\label{eq:schemeinidg}
 \langle w(0) - \Phi w(T) , v\rangle_H = \langle \xi_0 , v\rangle_H\mbox{ for all }v\in U_n
\end{equation}
and
\begin{multline}\label{eq:schememdg}
\int_{(m-1)k}^{mk}\langle\partial_b  w(t), v(t)\rangle_{ H }{\rm d}t + \langle w_0^{(m)}- w((m-1)k), v((m-1)k)\rangle_{ H }\\
+\int_{(m-1)k}^{mk} a(t, w(t), v(t)){\rm d}t 
  =\int_{(m-1)k}^{mk}\langle f(t),v(t)\rangle {\rm d}t\\
  \mbox{ for all }m=1,\ldots,N \mbox{ and } v \in \mathcal{P}^q(\mathbb{R};U_n).
\end{multline}

\begin{rem}
Note that, letting $q=0$, since $w(mk) = w_0^{(m)}$, Scheme \eqref{eq:schememdg} is reduced to
\begin{multline}\label{eq:schemeeulerimp}
 \langle w(mk)- w((m-1)k), v\rangle_{ H }
+\int_{(m-1)k}^{mk} a(t, w(mk), v){\rm d}t 
  =\int_{(m-1)k}^{mk}\langle f(t),v\rangle {\rm d}t\\
  \mbox{ for all }m=1,\ldots,N \mbox{ and } v\in U_n,
\end{multline}
which is the same expression as the one given for $\theta=1$  for the Euler implicit scheme in Section \ref{sec:eulerex}.
\end{rem}
We now observe that the space $W_n$ is isomorphic to $U_n^{m(q+1)+1}$. As in Section \ref{sec:eulerex}, we have that each element of $W_n$ defines a unique element of $V$, but this identification is again not injective, because of the initial value (it was the final value in  Section \ref{sec:eulerex}), and  we again do not distinguish the notation for the everywhere defined $w\in W_n$ and the element of $V_n$ it defines, where $V_n$ is the subspace of $V$ defined by
\begin{equation}\label{eq:defVndg}
V_n = \{v\in V: \ v_{|((m-1)k,mk)}\in \mathcal{P}^q(\mathbb{R};U_n)\hbox{ for all }m=1,\ldots,N\}.
\end{equation}
In order to put this setting into the framework of Sections \ref{sec:NCBNB} and \ref{sec:pbcont}, we define the discrete derivative of an element of $w\in W_n$ by modifying the broken derivative $\partial_b$.
 For any $i=0,\ldots,q$, let $\psi_i\in \mathcal{P}^q(\mathbb{R};\mathbb{R})$ be the polynomial such that
\begin{equation}\label{eq:defpsidg}
  \int_0^{1} \psi_i(s) s^i {\rm d} s = 1\hbox{ and }\int_0^{1} \psi_i(s)s^j{\rm d} s = 0\mbox{ for all } j \in\{0,\ldots,q\}\setminus\{i\} .
 \end{equation}
 In the case $q=0$,  we set $\psi_0(s) =1$ for all $s\in [0,1]$.
 Note that, for $i=0,\ldots,q$, the $q+1$ coefficients of the polynomials $\psi_i$ with degree $q$ are the coefficients of the $i$-th line or column of the inverse of the Gram matrix whose coefficients are equal to $\int_0^{1} s^{i+j}{\rm d} s = \frac 1 {i+j+1}$ for $j=0,\ldots,q$, and are elements of $\mathbb{Z}\setminus\{0\}$ (see Appendix \ref{sec:gram}). 
 For all $w\in W_n$, we then have the relation
 \begin{multline}\label{eq:defjumpdg}
  \frac 1 k\int_{(m-1)k}^{mk} \psi_0\Big(\frac {t-(m-1)k} k\Big)w(t){\rm d}t  = w_0^{(m)}  = \displaystyle\lim_{\substack{t\to (m-1)k \\ t>(m-1)k}} w(t)\mbox{ for all }m=1,\ldots,N-1.
 \end{multline}
 
 We then define $\partial: W_n\to V_n$,  by 
 \begin{multline}\label{eq:defqnevolstrdg}
\partial w(t) =  \partial_b w(t) + \frac  1 k\psi_0\Big(\frac {t-(m-1)k} k\Big) (w_0^{(m)}- w((m-1)k)),\\
 \hbox{ for a.e. }t\in ((m-1)k, mk),\hbox{ for all }m=1,\ldots,N,\hbox{ for all }w\in W_n.
 \end{multline}
Then $\widehat{\partial} w\in V'$ is again defined by \eqref{eq:defdhatfromd}, where $P^V_n$, the orthogonal projection on $V_n$ in $V$. It  is explicitely given  by 
\begin{multline}\label{eq:defpvndg}
P^V_n w (t) = \sum_{i=0}^q \left(\frac 1 k\int_{(m-1)k}^{mk} \Big(\frac {s-(m-1)k} k\Big)^i P^U_n w(s){\rm d}s \right) \psi_i\Big(\frac {t-(m-1)k} k\Big),\\ \hbox{ for a.e. }t\in((m-1)k,mk),\hbox{ for all }m=1,\ldots,N\hbox{ and for all } w\in W_n.
\end{multline}

Similarly to what has been done in Section~\ref{sec:eulerex}, we  define $X_n$ by 
\begin{equation}\label{eq:defXfromWdg}
X_n = \{(\widehat{\partial} w,w): w\in W_n\},
\end{equation}	
 and the mapping $W_n\to X_n$, $w\mapsto (\widehat{\partial} w,w)$ is again bijective. We now define the space $Y_n \subset Y := V\times H$ by
\begin{equation}\label{eq:defYdg}
Y_n = V_n\times U_n.
\end{equation}
The linear mapping $W_n\to Y_n$,  $w\mapsto (w ,w(0))$ is bijective as well. Keeping the definitions \eqref{eq:defb}  for $b$ and  \eqref{eq:defL} for $L$, we now have the following result.

\begin{lem}\label{lem:schemeexdg} Let $w\in W_n$. Let $x = (\widehat{\partial} w,w)\in X_n$. The following assertions are equivalent:
\begin{enumerate}
 \item[(i)] The element  $w$ satisfies \eqref{eq:schemeinidg}-\eqref{eq:schememdg}.
\item[(ii)] The element $x$ is solution of the following problem:
\begin{equation}\label{eq:schemeinZYdg}
x\in X_n,\ b(x,y) = L(y),\hbox{ for all }y\in Y_n.
\end{equation}
\end{enumerate}
\end{lem}
\begin{proof} This property holds by \eqref{eq:defqnevolstrdg} and \eqref{eq:defjumpdg}.
 
\end{proof}

Let us now prove that the conditions allowing to use the results of Sections \ref{sec:NCBNB} and \ref{sec:pbcont} are fulfilled.

\subsection{BNB estimates}
  
In order to prove the BNB estimate using Lemma \ref{lem:suffbnb}, we need some preliminary lemmas.

For any $w\in W_n$, we define $\widehat{\partial}_b w\in V'$ by
\begin{equation}\label{eq:defdhatfromdbroken}
 \langle \widehat{\partial}_b w,v\rangle_{V',V} = \langle \partial_b w,P^V_n  v\rangle_{\mathbb{H}}\mbox{ for all }w\in W_n \mbox{ and } v\in V.
\end{equation}
Note that, as in the definition of $\widehat{\partial} w$, $\widehat{\partial}_b w$ is different from $\partial_b w$ seen as an element of $V'$ by the standard embedding of $\mathbb{H}$ in $V'$, but that we have
\[
 \langle \widehat{\partial}_b w, v \rangle_{V',V} =\langle \partial_b w, v \rangle_{\mathbb{H}}\mbox{ for all } w,v\in W_n,
\]
as well as
\begin{equation}\label{eq:normeDUVprimebroken}
 \Vert \widehat{\partial}_b w\Vert_{V'} = \sup_{v\in V,\Vert v\Vert_V = 1}\langle \partial_b w,P^V_n  v\rangle_{\mathbb{H}} = \sup_{v\in W_n,\Vert v\Vert_V = 1}\langle \partial_b w,v\rangle_{\mathbb{H}}\mbox{ for all }w\in W_n.
\end{equation}

\begin{lem}\label{lem:normequivV}
There exists $\ctel{cte:equiv}>0$ only depending on $q$ such that
\[
\frac 1 {\cter{cte:equiv}}  \sum_{m=1}^N k\sum_{j=0}^q \Vert v_j^{(m)}\Vert_U^2 \le \Vert v \Vert_{V}^2 \le  \cter{cte:equiv} \sum_{m=1}^N k\sum_{j=0}^q \Vert v_j^{(m)}\Vert_U^2\mbox{ for all }v\in W_n.
 \]
\end{lem}
\begin{proof}
 We have
 \[
  \Vert v\Vert_{V}^2 = \sum_{m=1}^N k  \sum_{i=0}^q\sum_{j=0}^q \frac 1 {i+j+1} \langle v_i^{(m)},v_j^{(m)}\rangle_U,
 \]
 and the result holds from the fact that the Gram matrix $A$ of Appendix \ref{sec:gram} is symmetric definite positive.
\end{proof}

\begin{lem}\label{lem:controldg}
There exists $\ctel{cte:cq}>0$ only depending on $q$ such that
\begin{equation}\label{eq:majdbbyd}
\Vert \widehat{\partial}_b w\Vert_{V'} \le  \cter{cte:cq} \Vert \widehat{\partial} w\Vert_{V'} \mbox{ for all }w\in W_n,
\end{equation}
 and
\begin{equation}\label{eq:minoru}
\sum_{m=1}^N \Vert w_0^{(m)}- w((m-1)k)\Vert_H^2 \le  \cter{cte:cq} \Vert \widehat{\partial} w\Vert_{V'}\Vert w \Vert_V\mbox{ for all }w\in W_n.
\end{equation}
\end{lem}
\begin{proof}
We first observe that, for any $m=1,\ldots,q$, we have
\begin{multline*}
 \int_{(m-1)k}^{mk} \partial w(t)\psi_q\Big(\frac {t-(m-1)k} k\Big){\rm d}t =  \int_{(m-1)k}^{mk} \partial_b w(t)\psi_q\Big(\frac {t-(m-1)k} k\Big){\rm d}t \\
 +  \frac  1 k\int_{(m-1)k}^{mk}\psi_0\Big(\frac {t-(m-1)k} k\Big)\psi_q\Big(\frac {t-(m-1)k} k\Big){\rm d}t\  (w_0^{(m)}- w((m-1)k)),
\end{multline*}
where the polynomial $\psi_q$ is defined by \eqref{eq:defpsidg}. Since $\partial_b(w)(t)\in \mathcal{P}^{q-1}(\mathbb{R};U_n)$ \\ for a.e. $t\in ((m-1)k,mk)$, the first term on the right-hand-side of the above equation vanishes. Moreover, we have
\[
  \frac  1 k\int_{(m-1)k}^{mk}\psi_0\Big(\frac {t-(m-1)k} k\Big)\psi_q\Big(\frac {t-(m-1)k} k\Big){\rm d}t = \int_{0}^{1}\psi_0(s)\psi_q(s){\rm d}s = \psi_{q,0}\neq 0,
\]
by the results of Appendix \ref{sec:gram}. This leads to
\[
 w_0^{(m)}- w((m-1)k) = \frac 1 {\psi_{q,0}}\int_{(m-1)k}^{mk} \partial w(t)\psi_q\Big(\frac {t-(m-1)k} k\Big){\rm d}t.
\]
Let $v\in W_n$. We have
\[
 \langle \partial_b w,v \rangle_{\mathbb{H}} = \langle  \partial w,v \rangle_{\mathbb{H}} - \langle  \partial w,\widetilde{v} \rangle_{\mathbb{H}},
\]
where $\widetilde{v}\in W_n$ is such that, for a.e. $t\in ((m-1)k,mk)$, for any $m=1,\ldots,q$,
\[
 \widetilde{v}(t) = \frac 1 {\psi_{q,0}}\psi_q\Big(\frac {t-(m-1)k} k\Big) v_0^{(m)}.
\]
Since $\Vert \widetilde{v} \Vert_V\le \ctel{cte:equivvw}\Vert v \Vert_V$ by Lemma \ref{lem:normequivV}, we conclude \eqref{eq:majdbbyd}.

Moreover, we have
\[
 \sum_{m=1}^N \Vert w_0^{(m)}- w((m-1)k)\Vert_H^2 = \langle  \partial w,\widetilde{w}\rangle_{\mathbb{H}},
\]
with, for a.e. $t\in ((m-1)k,mk)$, for any $m=1,\ldots,q$,
\[
 \widetilde{w}(t) = (w_0^{(m)}- w((m-1)k))\frac 1 {\psi_{q,0}}\psi_q\Big(\frac {t-(m-1)k} k\Big),
\]
which implies \eqref{eq:minoru} again using Lemma \ref{lem:normequivV}.

\end{proof}

\begin{lem}\label{lem:estimnoruTHdg}  There exists $\ctel{cte:dgt}>0$, only depending on $q$,  $T$ and on $C_H$ (see \eqref{eq:contemb}),  such that 
\begin{equation}\label{eq:estimnoruTHdg}
\sup_{t\in(0,T]} \Vert w(t)\Vert_H^2 \le \cter{cte:dgt} (\Vert \widehat{\partial} w\Vert_{V'}^2 +\Vert  w \Vert_V^2)\mbox{ for all }w\in W_n.
\end{equation}
Therefore there exists $\rho>0$, only depending on $q$, $T$ and $C_H$, such that
\begin{equation}\label{eq:estimusupbynorminZdg} 
\sup_{t\in[0,T]}\Vert w(t)\Vert_H  \le  \rho \Big( \Vert \widehat{\partial} w \Vert_{V'}^2 + \Vert  w\Vert_V^2 +\Vert w(0)\Vert_H^2 + \Vert w(T)\Vert_H^2\Big)^{1/2}\mbox{ for all }w\in W_n.
\end{equation}
 \end{lem}
\begin{proof}
We notice that, in the case where $(m-1)k< s\le t \le mk$,
\[
 \Vert w(t)\Vert_H^2 -  \Vert w(s)\Vert_H^2 = 2\int_s^t  \langle \partial_b w(\tau),w(\tau)\rangle_H {\rm d}\tau 
 \le 2 \Vert \widehat{\partial}_b w\Vert_{V'}\Vert w \Vert_{V} \le (\cter{cte:cq} \Vert \widehat{\partial} w\Vert_{V'}^2 + \Vert w \Vert_{V}^2).
\]

We have, in the case where $s < mk \le (m'-1)k\le t \le m' k $ with $m < m'$, that
\[
  \Vert w(t)\Vert_H^2 -  \Vert w(m'k)\Vert_H^2 = - 2\int_t^{m'k}  \langle \partial_b w(\tau),w(\tau)\rangle_H {\rm d}\tau,
\]
\[
  \Vert w(m'k)\Vert_H^2 +\sum_{p=m+1}^{m'} \Vert w_0^{(p)} - w((p-1)k) \Vert_H^2-  \Vert w(mk)\Vert_H^2 = 2\int_{mk}^{m'k}  \langle \partial w(\tau),w(\tau)\rangle_H {\rm d}\tau,
\]
\[
  \Vert w(mk)\Vert_H^2 -  \Vert w(s)\Vert_H^2 =  2\int_s^{mk}  \langle \partial_b w(\tau),w(\tau)\rangle_H {\rm d}\tau.
\]
Adding the three above equalities, this yields
\begin{multline}\label{eq:majdif}
 \Vert w(t)\Vert_H^2 -  \Vert w(s)\Vert_H^2 = - 2\int_t^{m'k}  \langle \partial_b w(\tau),w(\tau)\rangle_H {\rm d}\tau+2\int_{mk}^{m'k}  \langle \partial w(\tau),w(\tau)\rangle_H {\rm d}\tau \\ +2\int_s^{mk}  \langle \partial_b w(\tau),w(\tau)\rangle_H {\rm d}\tau -\sum_{p=m+1}^{m'} \Vert w_0^{(p)} - w((p-1)k) \Vert_H^2.
\end{multline}
This implies, using Lemma \ref{lem:controldg}, that for any $s,t\in(0,T]$,
\[
 | \Vert w(t)\Vert_H^2 -  \Vert w(s)\Vert_H^2| \le \ctel{cte:psil} \Vert \widehat{\partial} w\Vert_{V'} \Vert w \Vert_V.
\]
This implies
\[
 \Vert w(t)\Vert_H^2 \le\Vert w(s)\Vert_H^2 +  \cter{cte:psil} \Vert \widehat{\partial} w\Vert_{V'} \Vert w \Vert_V.
\]
Integrating the preceding relation on $s\in (0,T)$ concludes the proof of \eqref{eq:estimnoruTHdg}. Then \eqref{eq:estimusupbynorminZdg} follows from the presence of the term $\Vert w(0)\Vert_H$ in the right hand side.
\end{proof}

\begin{lem}\label{lem:estimbnbndg}
 	The spaces $X_n$ and $Y_n$ defined by \eqref{eq:defXfromWdg} and \eqref{eq:defYdg} satisfy Hypothesis \eqref{eq:betabnbn} of Theorem \ref{thm:bnbnonconf} (without any CFL condition), with $\beta>0$ only depending on $q$, $T$, $C_H$, $\alpha$ and $M$.
 \end{lem}
 \begin{proof}
 We remark that
\[
 \langle\widehat{\partial} w,w \rangle_{V',V} =\sum_{m=1}^{N}\Big(\int_{(m-1)k}^{mk}\langle w'(t), w(t)\rangle_{ H }{\rm d}t + \langle w_0^{(m)}- w((m-1)k), w_0^{(m)}\rangle_{ H }\Big).
\]
Since we have
\[
 \int_{(m-1)k}^{mk}\langle w'(t), w(t)\rangle_{ H }{\rm d}t = \frac 1 2 \Vert w(mk)\Vert_H^2 - \frac 1 2 \Vert w_0^{(m)}\Vert_H^2,
\]
we get 
\[
2\langle\widehat{\partial} w,w \rangle_{V',V} =\Vert w(T )\Vert_H^2 + \sum_{m=1}^{N} \Vert w_0^{(m)}- w((m-1)k)\Vert_H^2-  \Vert w(0)\Vert_H^2,
\]
which yields
\[
 \langle\widehat{\partial} w,w \rangle_{V',V} \ge \frac 1 2  \Vert w(T)\Vert_H^2  - \frac 1 2\Vert  w(0)\Vert_H^2.
\]
We also notice that \eqref{eq:estimnoruTHdg} implies
\[
 \frac {\alpha^2}{12\, M^3} (\Vert  w \Vert_V^2 + \Vert \widehat{\partial} w\Vert_{V'}^2 )\ge \frac {\alpha^2}{12\, M^3\cter{cte:dgt}}\Vert w(T)\Vert_H^2.
\]
Adding the two previous inequalities yields  \eqref{eq:condlim}
with $\omega = \frac 1 2  +\frac {\alpha^2}{12\, M^3\cter{cte:dgt}}$ and $\delta = \frac {\alpha^2}{12\, M^3\cter{cte:dgt}}$ (hence no CFL condition will be required for satisfying the BNB estimate). We can now apply Lemma  \ref{lem:suffbnb}, letting $\widehat{X} = \{(\widehat{\partial} w,w,w(0),w(T)): w\in W_n\}$, since the quantities $\widehat{X}_i$ defined in the lemma satisfy $\widehat{X}_2=V_n$ and $ \widehat{X}_3 = U_n$. We get the existence of $\widehat{\beta}$,  only depending on $\alpha$, $M$, $T$ and $C_H$, such that
\[
	 \sup_{\substack{y\in Y_n\\ \Vert y\Vert_Y = 1}} b((\widehat{\partial} w,w),y) \ge \widehat{\beta} \Big( \Vert \widehat{\partial} w \Vert_{V'}^2 + \Vert  w\Vert_V^2 +\Vert w(0)\Vert_H^2 + \Vert w(T)\Vert_H^2\Big)^{1/2}\hbox{ for all }w\in W_n.
\]
Using Lemma \ref{lem:estimnoruTHdg} concludes the existence of $\beta$ (with the same dependencies) such that \eqref{eq:betabnbn} holds. 

 \end{proof}

 We can now give an error estimate for the discrete solution.
 
 \begin{thm}\label{thm:errestdg} Let  $\beta>0$ be given by Lemma \ref{lem:estimbnbndg}.
 Let $U_n\subset U$ be a finite dimensional subspace of $U$ and let $N\in\mathbb{N}$ be given. 
 Then the following holds. 
 
 For each $f\in V'$ and $\xi_0\in H$, the time discontinuous Galerkin Scheme \eqref{eq:schemeinidg}-\eqref{eq:schememdg} has a unique solution $w_n\in W_n$. Moreover
 \begin{multline}\label{eq:errestdg}
\Big(\Vert u' - \widehat{\partial}w_n\Vert_{V'}^2 +  \Vert u - w_n\Vert_{V}^2 + \sup_{t\in[0,T]} \Vert u(t) - w_n(t)\Vert_{H}^2\Big)^{1/2} \le \\
  (1+\frac {2M}{\beta})  \inf\{\Big(\Vert u' - \widehat{\partial}w\Vert_{V'}^2 +  \Vert u - w\Vert_{V}^2 + \sup_{t\in[0,T]} \Vert u(t) - w(t)\Vert_{H}^2\Big)^{1/2}: w\in W_n\},
 \end{multline}
 where $u$ is the unique solution of Problem \eqref{eq:5.1}.
 \end{thm}
 \begin{proof}
 Recalling the definitions of $Z$ and $Y$ from Section \ref{sec:extended}, we observe that $\beta$, given by Lemma \ref{lem:estimbnbndg}, satisfies  \eqref{eq:betabnbn}. Therefore, we apply Lemma \ref{lem:errest}, which provides \eqref{eq:errestdg}.
 \end{proof}

\subsection{Interpolation results}

 In the same way as above for the explicit Euler scheme, we now estimate the distance in $Z$ of a given function $w\in W$ from $W_n$ (recall that, again,  $W_n$ is not a subspace of $W$). 
 
 For a given $w\in H^{q+1}(0,T;U)$, we define the interpolation $w_n\in W_n$ by
 \begin{equation}\label{eq:definterpdg}
 w_{n,0}^{(0)} =P^U_n w(0)\hbox{ and }  w_{n,i}^{(m)}  =  \frac {k^i} {i!}P^U_n w^{[i]}((m-1)k)\mbox{ for all }m=1,\ldots,N \mbox{ and }  i=0,\ldots,q,
 \end{equation}
 where we denote by $w^{[i]}(t)$ the $i$th time derivative of $w$ for $i=0,\ldots,q+1$.
 Then
 \[
 w_n(0) =w_{n,0}^{(0)}\hbox{ and } w_n(t) = \sum_{i=0}^q \Big(\frac {t - (m-1)k} {k}\Big)^i w_{n,i}^{(m)}\mbox{ for all }m=1,\ldots,N\mbox{ and }t\in ((m-1)k,mk].
 \]

Recall that the following Taylor expansion formula holds:
\begin{multline}\label{eq:taylor}
 w(t) =\sum_{i=0}^q \frac {(t - (m-1)k)^i} {i!} w^{[i]}((m-1)k) + 
 \int_{(m-1)k}^t \frac {(s - (m-1)k)^q} {q!}w^{[q+1]}(s){\rm d}s \\
 \mbox{ for all }t\in ((m-1)k,mk].
\end{multline}
This provides, for any $m=1,\ldots,N$ and $t\in ((m-1)k,mk]$,
\begin{equation}\label{eq:taylorbis}
 P^U_n\circ w = w_n + T_n^{(m)} w\hbox{ with } T_n^{(m)} w(t) = \int_{(m-1)k}^t \frac {(s - (m-1)k)^q} {q!}P^U_n  w^{[q+1]}(s){\rm d}s.
\end{equation}

 We then have the following lemma.

 \begin{lem}\label{lem:interpinVdg} Let $w\in H^{q+1}(0,T;U)$ and let $w_n\in W_n$ be defined by \eqref{eq:definterpdg}.
 Then 
 \[
   \Vert w - P^V_n w\Vert_{V} \le \Vert w - w_n\Vert_{V} \le \delta_n(w) +   \frac {k^{q+1}}{q!} \Vert w^{[q+1]}\Vert_{V}.
 \]
 \end{lem}
 \begin{proof}
  Owing to \eqref{eq:taylorbis}, we have, for any $m=1,\ldots,N$ and $t\in ((m-1)k,mk]$,
  \[
   w(t) - w_n(t) = w(t) - P^U_n w(t) + \int_{(m-1)k}^t \frac {(s - (m-1)k)^q} {q!}P^U_n w^{[q+1]}(s){\rm d}s.
  \]
We notice that, owing to the Cauchy-Schwarz inequality, we have
\[
 \Vert \int_{(m-1)k}^t \frac {(s - (m-1)k)^q} {q!}P^U_n w^{[q+1]}(s){\rm d}s\Vert_U^2\le \frac {k^{2q+1}}{(q!)^2}\int_{(m-1)k}^{mk}\Vert w^{[q+1]}(s)\Vert_U^2{\rm d}s.
\]
Another application of  the Cauchy-Schwarz inequality concludes the proof of the lemma.  
 \end{proof}

 \begin{lem}\label{lem:cvschemedg}  
 We define $R^V~:~V'\to V$ as in Lemma  \ref{lem:rieszUV}.
Let $w\in H^{q+2}(0,T;U)$ and let $w_n\in W_n$ be defined by \eqref{eq:definterpdg}.
 Then, for $q=0$, we have
 
 \begin{equation}\label{eq:cvschemedgqz}
  \Vert w' - \widehat{\partial}w_n\Vert_{V'}  \le \delta_n(R^V w') +  k  \Vert w''\Vert_{V'} + C_H^2(\delta_n(w')+ k \Vert w''\Vert_{V}).
 \end{equation}
 
 and, for any $q\ge 1$, there exists  $\ctel{cte:csixq}$ only depending on $q$ such that
 \begin{equation}\label{eq:cvschemedg}
  \Vert w' - \widehat{\partial}w_n\Vert_{V'}  \le\delta_n(R^V w') +   \frac {k^{q+1}}{q!} \Vert w^{[q+2]}\Vert_{V'} +     \cter{cte:csixq} ( \delta_n(w')+k^{q} (\Vert w^{[q+2]}\Vert_{V'} + \Vert w^{[q+1]}\Vert_V)).
 \end{equation}
 \end{lem}
 \begin{proof}
 Let $\widehat{w_n'}\in V'$ be defined, for  a.e. $t\in(0,T)$, by
 \begin{equation}\label{eq:defrhhdg}
 \langle \widehat{w_n'}, v\rangle_{V',V} = \langle w', P^V_n  v\rangle_{V',V}\mbox{ for all }v\in V .
 \end{equation}
 We notice that
 \[
  \Vert w' - \widehat{\partial}w_n\Vert_{V'} \le \Vert w' - \widehat{w_n'}\Vert_{V'} + \Vert \widehat{w_n'} - \widehat{\partial}w_n\Vert_{V'}.
 \]
 We have, for any $v\in V$,

 \begin{multline*}
 \langle w' - \widehat{w_n'}, v\rangle_{V',V} = \langle w' , v\rangle_{V',V} -\langle w' , P^V_n  v\rangle_{V',V} \\ 
 =  \langle R^V w' , v\rangle_{V} -\langle R^V w' , P^V_n  v\rangle_{V} = \langle R^V w' - P^V_n  R^V w' ,v\rangle_{V}.
\end{multline*}
 This proves that
 \[
  \Vert  w' - \widehat{w_n'}\Vert_{V'} = \Vert R^V w' - P^V_n  R^V w'\Vert_{V}.
 \]
 We then apply Lemma \ref{lem:interpinVdg}, and we get
 \[
  \Vert  w' - \widehat{w_n'}\Vert_{V'} \le  \delta_n(R^V w') +   \frac {k^{q+1}}{q!} \Vert R^V w^{[q+2]}\Vert_{V} = \delta_n(R^V w') +   \frac {k^{q+1}}{q!} \Vert w^{[q+2]}\Vert_{V'}.
 \]
 This yields the first two terms of the right-hand-side of \eqref{eq:cvschemedg}.
 We also have  that, for any $v\in V$,
  \begin{multline*}
  \langle w' - \partial w_n,P^V_n  v\rangle_{\mathbb{H}}\le \Vert w' - \partial w_n\Vert_{\mathbb{H}}\Vert P^V_n  v \Vert_{\mathbb{H}}
  \\ \le C_H \Vert w' - \partial w_n\Vert_{\mathbb{H}}\Vert P^V_n  v \Vert_{V} \le  C_H \Vert w' - \partial w_n\Vert_{\mathbb{H}}\Vert v\Vert_{V}.
\end{multline*}
We then get 
  \begin{equation*}
  \Vert  \widehat{w_n'} - \widehat{\partial}w_n\Vert_{V'} =\sup_{v\in V,\Vert v\Vert_V = 1}  \langle w' - \widehat{\partial} w_n,P^V_n  v\rangle_{V',V}
  \le  C_H \Vert w' - \partial w_n\Vert_{\mathbb{H}}.
\end{equation*}
 We then apply
 \begin{multline*}
  \Vert w' - \partial w_n\Vert_{\mathbb{H}} \le \Vert w' - P^U_n w'\Vert_{\mathbb{H}}+ \Vert P^U_n w' - \partial w_n\Vert_{\mathbb{H}}
  \\ \le C_H(\Vert w' - P^U_n w'\Vert_{V}+ \Vert P^U_n w' - \partial w_n\Vert_{V}).
 \end{multline*}

 We can write, for a.e. $t\in((m-1)k,mk)$ and all $m=1,\ldots,N$, that
 \begin{multline*}
 P^U_n w'(t) = \sum_{i=1}^{q} \frac {(t - (m-1)k)^{i-1}} {(i-1)!} P^U_n w^{[i]}((m-1)k) + 
 \int_{(m-1)k}^t \frac {(s - (m-1)k)^{q-1}} {(q-1)!} P^U_n w^{[q+1]}(s){\rm d}s \\
 = \partial_b w_n(t) +  \int_{(m-1)k}^t\frac {(s - (m-1)k)^{q-1}} {(q-1)!} P^U_n w^{[q+1]}(s){\rm d}s.
 \end{multline*}
 
 Let us first consider the case $q>0$. We write
 \[
   \Vert P^U_n w' - \partial w_n\Vert_{V}   \le \Vert P^U_n w' - \partial_b w_n\Vert_{V}+ \Vert \partial w_n - \partial_b w_n\Vert_{V}.
 \]
 This yields
 \[
  \Vert P^U_n w' - \partial_b w_n\Vert_{V} \le \frac {k^q}{(q-1)!} \Vert w^{[q+1]}\Vert_V.
 \]
On the other hand, we have, using
\[
 w_n((m-1)k) =  \sum_{i=0}^q \frac {k^i} {i!} P^U_n w^{[i]}((m-2)k)
\]
and  $v_0^{(m)} = P^U_n w((m-1)k)$
 \begin{multline*}
 \Vert \partial w_n - \partial_b w_n\Vert_{V}^2 \le C_\psi \sum_{m=1}^N \frac 1 k \Vert v_0^{(m)} - w_n((m-1)k)\Vert_U^2 \\
 = C_\psi \sum_{m=1}^N  \frac 1 k \Vert \int_{(m-2)k}^{(m-1)k} \frac {(s - (m-2)k)^q} {q!}w^{[q+1]}(s){\rm d}s\Vert_U^2 
 \le C_\psi \frac {k^{2q}} {(q!)^2}  \Vert w^{[q+1]}\Vert_V^2.
 \end{multline*}
 Gathering these results provides \eqref{eq:cvschemedg}.
 
 \medskip
 
 Let us now consider the case $q=0$. The reasoning of Lemma \ref{lem:cvschemeex} applies to this case, and we again get 
 \[
  \int_{mk}^{(m+1)k}\Vert P^U_n w'(t) - \partial w_n(t)\Vert_U^2{\rm d}t\le k^2\int_{mk}^{(m+1)k}\Vert w''(\tau)\Vert_U{\rm d}\tau.
 \]
  \end{proof}

 \begin{lem}\label{lem:interpinVpdg}  Let $w\in H^{q+2}(0,T;U)$ and let $w_n\in W_n$ be defined by \eqref{eq:definterpdg}.
 Then 
 \[
 \sup_{t\in[0,T]}\Vert w_n(t) - w(t)\Vert_H \le \ctel{cte:cseptq} (k^{q+1}(\Vert w^{[q+1]}\Vert_V+\Vert w^{[q+2]}\Vert_{V})+ \delta_n(w) + \delta_n(w')).
 \]
 \end{lem}
 \begin{proof}
 For $m=1,\ldots,N$ and all $t\in((m-1)k,mk]$, we have
 \begin{multline*}
  \Vert w_n(t) - w(t)\Vert_H  \le \Vert w_n(t) - P^U_n w(t)\Vert_H + \Vert P^U_n w(t) - w(t)\Vert_H \\
  \le C_H \Vert w_n(t) - P^U_n w(t)\Vert_U + C_T(\delta_n(w) + C_H\delta_n(w')).
 \end{multline*}
 
 Using \eqref{eq:taylorbis}, we have 
\[
 \Vert w_n(t) - P^U_n w(t)\Vert_U \le \int_{(m-1)k}^t \frac {(s - (m-1)k)^q} {q!}\Vert w^{[q+1]}(s)\Vert_U{\rm d}s.
\]
 Applying \eqref{eq:poincmoy} to $w^{[q+1]}(s)$ instead of $w'$, we obtain
\[
 \Vert w_n(t) - P^U_n w(t)\Vert_U \le \frac {k^{q+1}} {(q+1)!} (\frac 1 {\sqrt{T}} \Vert w^{[q+1]}\Vert_V+ \sqrt{T}\Vert w^{[q+2]}\Vert_V).
\]
 
This concludes the proof of the lemma.  
 \end{proof}

 Putting together the results of this section, we now show that we can approximate each function of $W$ by elements of $W_n$. For this we now consider a sequence of finite-dimensional subspaces $(U_n)_{n\in\mathbb{N}}$ of $U$ such that
 \[
  \displaystyle\lim_{n\to\infty} U_n = U.
 \]
This is equivalent to saying
 \[
  \displaystyle\lim_{n\to\infty} \Vert  u -  P^U_n u\Vert_U =0\hbox{ for all }u\in W,
 \]
 where $P^U_n$ denotes the orthogonal projection of $U$ to $U_n$.

 Let $N_n\in \mathbb{N}$ such that $\displaystyle\lim_{n\to \infty} N_n = \infty$. For each $n\in\mathbb{N}$, let $k_n = \frac {T}{N_n}$ be the time step and let
 \[
  W_n=\{w:[0,T]\to U_n: w_{|((m-1)k,mk]}\in\mathcal{P}^q(\mathbb{R};U_n)\hbox{ for }m=0,\ldots,N_n - 1\}.
 \]
 Recall that $W = H^1(0,T;U')\cap V $. We have the following approximation results.
 \begin{thm}\label{thm:3.9dg}
  Let $w\in W$. Then there exists $w_n\in W_n$ such that
  \[
   \Vert w' - \widehat{\partial}w_n\Vert_{V'}^2 +  \Vert w - w_n\Vert_{V}^2 + \sup_{t\in[0,T]} \Vert w(t) - w_n(t)\Vert_{H}^2 \to 0 \hbox{ as }n\to\infty. 
  \]
This proves that $X\subset \displaystyle\lim_{n\to\infty} X_n$, letting $X$ be defined by \eqref{eq:defhatw} and $X_n$ be defined by \eqref{eq:defXfromWdg}.
 \end{thm}
 \begin{proof}
 Let $w\in W$ and let $\varepsilon>0$. There exists $v\in H^{q+2}(0,T;U)$ such that
  \[
    \Vert w' - v'\Vert_{V'} +  \Vert w - v\Vert_{V}\le \varepsilon.
  \]
  By \eqref{eq:embWinH} this implies that $\sup_{t\in[0,T]}\Vert w(t) - v(t)\Vert_{H}\le C_H \varepsilon$.
  Let $v_n\in W_n$ be given from $v$ via \eqref{eq:definterpdg}. Then, by Lemma \ref{lem:interpinVdg},  $\Vert v_n - v\Vert_V\to 0$ as $n\to\infty$. By Lemma \ref{lem:cvschemedg}, $\displaystyle\lim_{n\to\infty}\Vert v' - \widehat{\partial}v_n\Vert_{V'} = 0$ and, by Lemma \ref{lem:interpinVpdg},  $\displaystyle\lim_{n\to\infty}\sup_{t\in[0,T]}\Vert v(t) - v_n(t)\Vert_{H} = 0$.

  Since $\varepsilon>0$ is arbitrary, the proof is complete.
 \end{proof}

 \subsection{Conclusion: optimal convergence}
 
 Let $(U_n)_{n\in\mathbb{N}}$ be a sequence of finite dimensional subspaces of $U$ such that
 \[
 U = \displaystyle\lim_{n\to\infty}U_n.
\]
Let $(N_n)_{n\in\mathbb{N}}$ be a sequence of integers which tends to infinity.
 Recall that $W_n$ is the space of piecewise polynomial functions defined in Subsection~\ref{sub:51}. 
 
 \begin{thm}\label{thm:cvdg} For given $f\in V'$ and $\xi_0\in H$, let $u$ be the unique solution of Problem \eqref{eq:5.1}.
 
Then we have
\begin{equation}\label{eq:convXYdg}
 X = \displaystyle\lim_{n\to\infty}X_n\hbox{ and }Y = \displaystyle\lim_{n\to\infty}Y_n,
\end{equation}
where $X$, $Y$ and $Z$ are defined in Section  \ref{sec:extended} and $X_n$ and $Y_n$ are defined by \eqref{eq:defXfromWdg} and \eqref{eq:defYdg} from $W_n$ and $V_n$ defined in Section \ref{sub:51}. 
 Moreover, letting  $w_n\in W_n$ be the unique solution of the time discontinuous Galerkin Scheme \eqref{eq:schemeinidg}-\eqref{eq:schememdg} for each $n\in \mathbb{N}$, we have
 \[
  \displaystyle\lim_{n\to\infty} \Big(\Vert u' - \widehat{\partial}w_n\Vert_{V'} +  \Vert u - w_n\Vert_{V} + \sup_{t\in[0,T]} \Vert u(t) - w_n(t)\Vert_{H}\Big)= 0.
 \]
 \end{thm}
 \begin{proof}
 Relations \eqref{eq:convXYdg} are immediate consequences of Theorem \ref{thm:errestdg}. Since, in addition, Theorem \ref{thm:5.3} implies that Condition \eqref{eq:denserange} is fulfilled, we can apply Theorem \ref{thm:bnbnonconf}, which concludes the proof of the theorem.
 \end{proof}

\appendix
\section{Technical lemmas in Hilbert spaces}\label{sec:hilbert}
We present some technical lemmas involving operators on Hilbert spaces.

\begin{lem}\label{lem:zigoto}
Let $V$ be a Hilbert space and let  $A~:~V\to V$ be a $M$-continuous and $\alpha$-coercive operator for $M\ge 1$ and $\alpha>0$, which means that 
\begin{equation}\label{eq:3.7m}
\Vert A v\Vert_{V} \le M \|v\|_V\mbox{ for all }v\in V,
\end{equation}  
and
\begin{equation}\label{eq:3.7n}
 \langle A v, v\rangle_{V} \ge \alpha \|v\|_V^2\mbox{ for all }v\in V.
\end{equation}  
Then
 \begin{equation}\label{eq:3.8n}
 \| w + A v\|_{V}^2\geq 2\alpha\langle w,v\rangle_{V}+ \frac 1 3 (\frac \alpha M)^3 (\|w \|_{V}^2 +\|v\|_V^2) \mbox{ for all }v,w\in V .
 \end{equation} 
\end{lem}
\begin{proof}
 Consider $A_{\hbox{\tiny +}}:=\frac{A+A^*}{2}$ and $A_{\hbox{\tiny -}}:=\frac{A-A^*}{2}$. Then, for any $v\in V$,
 \[\langle A_{\hbox{\tiny +}} v,v\rangle_V=\frac{1}{2}\langle Av,v\rangle_V + \frac{1}{2}\langle A^*v,v\rangle_V=\langle Av,v\rangle_V ,\,\, v\in V,\] 
 and thus $\langle A_{\hbox{\tiny +}} v,v\rangle\geq \alpha \|v\|_V^2, \,\, v\in V$. It follows that the selfadjoint operator $A_{\hbox{\tiny +}}$ is positive and invertible. Let $S$ be the (invertible and positive) square root of $A_{\hbox{\tiny +}}$. 
 Since $\| Sv\|_V^2=\langle A_{\hbox{\tiny +}}v,v\rangle_V$, we get 
 \begin{equation}\label{eq:pf1}
 \|Sv\|_V\geq \sqrt{\alpha} \|v\|_V,\,\, v\in V.
 \end{equation}
 Since $\|Sv\|\leq \|A_{\hbox{\tiny +}}\|^{1/2}\|v\| \le \sqrt{M}\|v\| $, we also get 
 \begin{equation}\label{eq:pf2}
  \|S^{-1} v\|   \geq \frac{1}{\sqrt{M}} \|v\|,\, \, v\in V.  
 \end{equation} 
  
 In addition, note that 
 \[    \langle A_{\hbox{\tiny -}}v,v\rangle_V=\langle v,A^*_{\hbox{\tiny -}} v\rangle_V =-\langle v,A_{\hbox{\tiny -}}v\rangle_V,\,\, v\in V,\]
 and therefore 
 \begin{equation}\label{eq:pf3}
  \langle A_{\hbox{\tiny -}} v,v\rangle_V =0,\,\, v\in V.
 \end{equation}
 Using \eqref{eq:pf1} we first get, for all $v\in W$, 
 \begin{eqnarray*}
  \|w + Av\|^2\ & \geq &  \alpha (\|S^{-1}w + S^{-1}A_{\hbox{\tiny +}} v+ S^{-1}A_{\hbox{\tiny -}}v\|_V^2)  \\
   & = &   \alpha( \|S^{-1}w\|_V^2 + \|S^{-1}A_{\hbox{\tiny +}}v\|_V^2 +\|S^{-1}A_{\hbox{\tiny -}}v\|_V^2 )  \\
    &  & + 2\alpha \langle S^{-1}w, S^{-1}A_{\hbox{\tiny +}} v\rangle_V + 2\alpha \langle S^{-1}A_{\hbox{\tiny +}}v, S^{-1}A_{\hbox{\tiny -}} v\rangle_V \\
      &  & + 2\alpha \langle S^{-1}w, S^{-1}A_{\hbox{\tiny -}} v\rangle_V .  
 \end{eqnarray*}
Note that 
\[ \langle S^{-1}A_{\hbox{\tiny +}}v, S^{-1}A_{\hbox{\tiny -}} v\rangle_V = \langle S^{-2}A_{\hbox{\tiny +}}v, A_{\hbox{\tiny -}} v\rangle_V =  \langle v, A_{\hbox{\tiny -}} v\rangle_V =0\mbox{ by }\eqref{eq:pf3}.\]
Moreover
\[  \langle S^{-1}w, S^{-1}A_{\hbox{\tiny +}} v\rangle_V=  \langle w, S^{-2}A_{\hbox{\tiny +}} v\rangle_V = \langle w,  v\rangle_V.\]
In addition, 
\[\|S^{-1}A_{\hbox{\tiny +}}v\|_V^2=\langle S^{-1}A_{\hbox{\tiny +}}v,S^{-1}A_{\hbox{\tiny +}}v\rangle_V=\langle A_{\hbox{\tiny +}}v,v\rangle_V=\|Sv\|_V^2.\]
It follows that  
\begin{eqnarray*}
	\|w + Av\|^2\ & \geq &  2\alpha\langle w,  v\rangle_V+  \alpha (\|S^{-1}w\|_V^2 + \|Sv\|_V^2 +\|S^{-1}A_{\hbox{\tiny -}}v\|_V^2 )  \\
	&  & + 2\alpha  \langle S^{-1}w, S^{-1}A_{\hbox{\tiny -}} v\rangle_V .  
\end{eqnarray*}
Now we use the celebrated Peter--Paul inequality combined with the Cauchy-Schwarz inequality to estimate, for all $\gamma>0$,   
\begin{eqnarray*}
2 \left|  \langle S^{-1}w, S^{-1}A_{\hbox{\tiny -}} v\rangle_V \right|& \leq &  2 \|S^{-1}w\|_V \|S^{-1}A_{\hbox{\tiny -}}v\|_V\\
 & \leq & \gamma \|S^{-1}w\|_V^2 +\frac{1}{\gamma} \|S^{-1}A_{\hbox{\tiny -}}v\|_V^2.
\end{eqnarray*}
We then get, choosing $\gamma <1$,
\[
	\|w+Av\|^2\  \geq   2\alpha\langle w,  v\rangle_V+ \alpha (1-\gamma)\|S^{-1}w\|_V^2 +\alpha (1-1/\gamma )\|S^{-1}A_{\hbox{\tiny -}}v\|_V^2
		+\alpha \|Sv\|_V^2.\]
We have
\[
 \|S^{-1}A_{\hbox{\tiny -}}v\|_V \le \frac {M} {\sqrt{\alpha}} \|v\|_V,
\]
which gives
\[
	\|w+Av\|^2\  \geq   2\alpha\langle w,  v\rangle_V+ \alpha (1-\gamma)\|S^{-1}w\|_V^2 + M^2(1-1/\gamma )\|v\|_V^2
		+\alpha \|Sv\|_V^2.\]
Let $\gamma=\frac{1}{1+s}$ where $s>0$ will be fixed later.  Then $1-\gamma=\frac{s}{1+s}$ and $1-1/\gamma = -s$.  Then, by \eqref{eq:pf1} and \eqref{eq:pf2}, it follows that 
\[ 	\|w+Av\|^2 \geq 2\alpha \langle w,  v\rangle_V+ \frac{s}{1+s}\frac{\alpha}{M}\|w\|_V^2 +\alpha^2\|v\|_V^2 -s M^2 \|v\|_V^2.    \]
Then we choose $s	=\frac{\alpha^2}{2 M^2}$ and we obtain 
\[  	\|w+Av\|^2 \geq 2\alpha\langle w,  v\rangle_V+  \beta (\|w\|_V^2 +\|v\|_V^2),   \]
where, using $\alpha\le M$ and $1\le M$,
\[
 \beta	=\min \left\{   \frac{\alpha^2}{2}, \frac{\alpha^2}{\alpha^2 +2M^2}\frac{\alpha}{M}   \right\} = \frac{\alpha^2}{\alpha^2 +2M^2}\frac{\alpha}{M}  \ge \frac 1 3 \Big(\frac \alpha M\Big)^3.
\]
 \end{proof}

\begin{lem}\label{lem:peterpaul2}
Let $H$ be a Hilbert space, and let $\Phi~:~H\to H$ be a contraction (which means that $\Vert\Phi\Vert\le 1$). Let $a>0$ and $b\in [0,a]$ be given reals such that $\gamma := a - b\Vert\Phi\Vert^2>0$. The following estimate holds
 \begin{equation}\label{eq:peterpaul2}
  a\| w\|_{H}^2 - b \| v\|_{H}^2 +  \frac {9 a^2}{\gamma}\| v - \Phi w\|_{H}^2 \geq \frac {\gamma} 3 (\| w\|_{H}^2+\| v\|_{H}^2)\mbox{ for all }v,w\in H .
 \end{equation} 
\end{lem}
\begin{proof}
Let $v,w\in H$ be given.
 By the Peter-Paul inequality, we have for any $\mu>0$,
 \[
  -2 \langle v, \Phi w\rangle_H \ge -\mu \| \Phi w\|_{H}^2- \frac 1 \mu \| v\|_{H}^2.
 \]
Choosing $\mu>1$ and using that $\Phi$ is a contraction, this implies that
\[
 \Vert v-{\Phi} w\Vert_H^2 \ge \Vert v\Vert_H^2(1 - \frac 1 \mu) - \Vert {\Phi} w\Vert_H^2 (\mu - 1)\ge\Vert v\Vert_H^2(1 - \frac 1 \mu) - \Vert w\Vert_H^2\Vert\Phi\Vert^2 (\mu - 1).
\]
Let  $\beta := b\Vert\Phi\Vert^2 \in [0,a)$ and $ \mu := \frac {\beta + 2 a}{2 \beta + a}\in (1,+\infty)$. Let $\theta>0$ and $\alpha>0$ be such that
\[
 \theta (1 - \frac 1 \mu) = b + \alpha
\]
and
\[
 \theta (\mu - 1)\Vert\Phi\Vert^2 =  a - \alpha.
\]
Using $\Vert\Phi\Vert^2\le 1$, this system is satisfied for
\[
 \alpha = \frac {(a-\beta)(a+\beta)}{2\beta + a + \Vert\Phi\Vert^2(2a+\beta)}\in [ \frac {\gamma}{3}, 2 a].
\]
Using the preceding equation and $b\le a$, we get
\[
\theta =  \frac {2a+\beta}{a-\beta}(b +\alpha) \le \frac {9 a^2}{\gamma}.
\]
We then obtain
\begin{multline*}
 a\| w\|_{H}^2 - b \| v\|_{H}^2 +  \frac {9 a^2}{\gamma}\| v - \Phi w\|_{H}^2 \ge a\| w\|_{H}^2 - b \| v\|_{H}^2 +  \theta\| v - \Phi w\|_{H}^2 \\
 \ge (a - \theta (\mu - 1)\Vert\Phi\Vert^2)\| w\|_{H}^2+(\theta (1 - \frac 1 \mu) -b) \| v\|_{H}^2  = \alpha(\| w\|_{H}^2+\| v\|_{H}^2)\ge \frac {\gamma} 3 (\| w\|_{H}^2+\| v\|_{H}^2).
\end{multline*}
The above inequality gives \eqref{eq:peterpaul2}.
\end{proof}

We now give a sufficient condition for proving the BNB condition for the different schemes studied here.
\begin{lem}\label{lem:suffbnb}
Let $V$ and $H$ be Hilbert spaces. Let $\widehat{Z}$ and $Y$ be the Hilbert spaces defined by $\widehat{Z} = V'\times V\times H\times H$ and $Y = V\times H$. Let ${\mathcal A}:V\to V'$ be a continuous and coercive linear operator in the sense that there exist $0< \alpha\le M$ such that
\begin{equation}\label{eq:continuousappendix}
 \Vert \mathcal{A} v\Vert_{V'}\le M \Vert v\Vert_{V}\mbox{ for all }v\in V,
\end{equation}
and
\begin{equation}\label{eq:coerciveappendix}
 \alpha \Vert v\Vert_{V}^2\le \langle \mathcal{A} v,v\rangle_{V',V}\mbox{ for all }v\in V.
\end{equation}
Let $\Phi:H\to H$ be a linear operator such that $\Vert \Phi\Vert\le 1$. We define $\widehat{b}~:~\widehat{Z}\times Y\to \mathbb{R}$ by
 \begin{equation}\label{eq:defblemme} 
 \widehat{b}((z_1,z_2,z_3,z_4),(y_1,y_2)) = \langle z_1 + \mathcal{A} z_2,y_1\rangle_{V',V} + \langle z_3 - \Phi z_4,y_2\rangle_{H},
 \end{equation}
for all $(z_1,z_2,z_3,z_4)\in \widehat{Z}$ and for all $(y_1,y_2)\in Y$.

 Let $\widehat{X}\subset \widehat{Z}$ be a subspace of $\widehat{Z}$. We define the spaces $\widehat{X}_i$ for $i=1,2,3,4$ by
 \[
  \widehat{X}_i = \overline{\{ x_i: \ x\in \widehat{X}\} },\hbox{ where }x_i\hbox{ is the $i$-th component of }x\in \widehat{Z}.
 \]
 Then $\widehat{X}_1\subset V'$, $\widehat{X}_2\subset V$, $\widehat{X}_3\subset H$ and $\widehat{X}_4\subset H$ are Hilbert spaces. We assume that
 \begin{equation}\label{eq:hypxun}
  \langle x_1,v\rangle_{V',V} = \langle x_1,P_{2}^V v\rangle_{V',V}\mbox{ for all }x_1\in \widehat{X}_1 \mbox{ and }v\in V,
 \end{equation}
where we denote by $P_{2}^V$ the orthogonal projection on $\widehat{X}_2$ in $V$. 
Assume that there exist $\omega>0$ and $\delta >0$ such that
\begin{equation}\label{eq:condlim}
 \langle x_1,x_2\rangle_{V',V} +\frac {\alpha^2} {12\, M^3}(\Vert  x_2\Vert_V^2 + \Vert x_1\Vert_{V'}^2)\ge \mu\Vert x_4\Vert_H^2 -\nu \Vert x_3\Vert_H^2\mbox{ for all }x\in \widehat{X},
\end{equation}
for some $\mu\in (0,\omega]$, $\nu\in [0,\mu]$ with $\mu - \nu\Vert\Phi\Vert^2 \ge \delta$.
 Then,  there exists $\widehat{\beta}>0$, only depending on $\alpha$, $M$, $\omega$ and $\delta$ (and not on $\mu$, $\nu$ and $\Vert\Phi\Vert$) such that
 \begin{equation}\label{eq:bnbdisc}
  \sup_{y\in \widehat{X}_2\times \widehat{X}_3, \Vert y\Vert_Y = 1} \widehat{b}(x,y) \ge  \widehat{\beta} \Vert x\Vert_{\widehat{Z}}\mbox{ for all } x\in \widehat{X}.
 \end{equation}
\end{lem}

\begin{proof}
By the Riesz representation theorem, there exists an operator $R~:~\widehat{X}_1\to \widehat{X}_2$ 
\begin{equation}\label{eq:defdzb}
\langle  x_1, y_2\rangle_{V',V} = \langle R x_1,y_2\rangle_V\mbox{ for all } x_1\in \widehat{X}_1 \mbox{ and } y_2\in \widehat{X}_2.
\end{equation}

We then have, in view of \eqref{eq:hypxun},
\begin{equation}\label{eq:dzb}
\Vert  x_1\Vert_{V'} = \Vert R x_1 \Vert_{V}\mbox{ for all }x_1\in \widehat{X}_1.
\end{equation}
We also define  the operator $A~:~\widehat{X}_2\to  \widehat{X}_2$ by
\begin{equation}\label{eq:defopA}
\langle A x_2,y_2\rangle_V :=\langle \mathcal{A} x_2, y_2\rangle_{V',V}\mbox{ for all }x_2\in \widehat{X}_2\mbox{ and }y_2\in \widehat{X}_2  .
\end{equation}
Then $A$ is $\alpha$-coercive and $\Vert A \Vert\le M$.
Let $ P^H_{3}:H\to \widehat{X}_3\subset H$ be the orthogonal projection on $\widehat{X}_3$.
Then 
\begin{equation}\label{eq:bnbX0}
	\mathcal{N}(x)^2
	= \Vert R x_1 + A x_2\Vert_{V}^2 + \Vert P^H_{3}(x_3 - \Phi x_4)\Vert_H^2\mbox{ for all }x\in \widehat{X},
	\end{equation}
where
\[
 \mathcal{N}(x) = \sup_{y\in \widehat{X}_2\times \widehat{X}_3, \Vert y\Vert_Y = 1} \widehat{b}(x,y)\mbox{ for all }x\in \widehat{X}.
\]
We then obtain, for $\theta>0$ to be chosen later

\begin{equation}\label{eq:bnbX}
		\mathcal{N}(x)^2
	\ge \frac 1 {\max(1,\theta)} \Vert R x_1 + A x_2\Vert_{V}^2 + \frac{\theta}{\max(1,\theta)} \Vert P^H_{3}(x_3 - \Phi x_4)\Vert_H^2\mbox{ for all }x\in \widehat{X}.
	\end{equation}

 We apply Lemma \ref{lem:zigoto} to $V = \widehat{X}_2$ with the scalar product of $V$ and obtain, by \eqref{eq:defdzb},
\[
 \Vert  R x_1 +  A x_2\Vert_{V}^2  \ge 2\alpha \langle R x_1,x_2\rangle_{V} + \frac 1 3 (\frac \alpha M)^3(\Vert  x_2\Vert_V^2 + \Vert R x_1 \Vert_{V}^2 ).
\]
Now \eqref{eq:condlim} implies that
\[
 \Vert  R x_1 +  A x_2\Vert_{V}^2  \ge 2\alpha(\mu\Vert x_4\Vert_H^2  - \nu\Vert  x_3\Vert_H^2)+ (\frac 1 3 (\frac \alpha M)^3-2\alpha\frac {\alpha^2} {12\, M^3}) (\Vert  x_2\Vert_V^2 + \Vert x_1 \Vert_{V'}^2 ).
\]
Together with \eqref{eq:bnbX}, this yields
\[
 \max(1,\theta)\mathcal{N}(x)^2 \ge 2\alpha(\mu\Vert x_4\Vert_H^2  - \nu\Vert  x_3\Vert_H^2)+ \frac {\alpha^3} {6\, M^3} (\Vert  x_2\Vert_V^2 + \Vert x_1 \Vert_{V'}^2 ) 
  + \theta\Vert P^H_{3}(x_3 - \Phi x_4)\Vert_H^2.
\]
Noting that $ P^H_{3}(x_3 - \Phi x_4) = x_3 - P^H_{3}\Phi x_4$ and that $\Vert P^H_{3}\circ\Phi \Vert\le \Vert\Phi\Vert \le 1$, we now use Lemma \ref{lem:peterpaul2}, in which we define $a= 2\alpha\mu$, $b= 2\alpha\nu$. In this way we obtain $\gamma \ge 2\alpha(\mu - \nu\Vert\Phi\Vert^2)\ge 2\alpha\delta$. If we set
$\theta = \frac {9 a^2}{\gamma} \le \frac {18\alpha\mu^2}{\mu - \nu\Vert\Phi\Vert^2}\le  \frac {18\alpha\omega^2}{\delta}$, we get
\[
 2\alpha(\mu\Vert x_4\Vert_H^2  - \nu\Vert  x_3\Vert_H^2) + \theta \Vert  P^H_{3}(x_3 - \Phi x_4)\Vert_H^2
 \ge  \frac{\gamma} {3} (\Vert x_3\Vert_H^2+\Vert x_4\Vert_H^2).
\]
This in turn gives
\[
\max(1,\frac {18\alpha\omega^2}{\delta}) \mathcal{N}(x)^2
 \ge   \frac {\alpha^3} {6\, M^3}  (\Vert  x_1\Vert_{V'}^2 + \Vert x_2\Vert_{V}^2) + \frac{2\alpha\delta} {3}(\Vert x_3\Vert_H^2+\Vert x_4\Vert_H^2),
\]
which leads to \eqref{eq:bnbdisc}.
\end{proof}

\section{Inverse of the Gram matrix}\label{sec:gram}

The following lemma is used in Section \ref{sec:dg}.
\begin{lem}\label{lem:gramat}
Let $q\in \mathbb{N}$. We consider the scalar product $\langle f,g\rangle = \int_0^1 f(x)g(x) {\rm d}x$ on the space $\mathcal{P}^q(\mathbb{R};\mathbb{R})$ with its natural basis. Let $A$ be the resulting square Gram matrix  with side $q+1$ defined by $A_{ij} = \int_0^1 x^{i+j}{\rm d}x$, therefore equal to $A = (\frac 1 {i+j+1})_{i,j=0,\ldots,q}$. Then $A$ is symmetric positive definite and
the coefficients of the inverse of $A$ are the nonzero integers  given by
\begin{multline}\label{eq:defgramminusone}
 (A^{-1})_{ij} =  (-1)^{i+j}\sum_{k=\max(i,j)}^{q}(2k+1)\mathcal{C}( k ,i)\mathcal{C}( k+i ,i)\mathcal{C}( k ,j)\mathcal{C}( k+j ,j)\\
 \mbox{ for all } i,j=0,\ldots,q,
\end{multline}
where we denote the binomial coefficients by
\[
  \mathcal{C}( k ,j) = \frac{ k!}{j!(k-j)!}\mbox{ for all }k\in\mathbb{N}\mbox{ and } j=0,\ldots,k .
\]
\end{lem}
\begin{proof}
We consider the Legendre polynomials defined by
\[
 P_k(x) = \frac {{\rm d}^k}{{\rm d}x^k}\Big(\frac {\sqrt{2k+1}}{k!} (x(1-x))^k\Big)\mbox{ for all }x\in\mathbb{R}\mbox{ and }  k=0,\ldots,q,.
\]
Denoting these polynomials by $P_k(x) =\sum_{i=0}^k P_{ik} x^i$, one defines the upper triangular matrix $P$ with side $q+1$ such that $P_{ik}$ is the coefficient of $P$ at the column $k=0,\ldots,q$ and at the line $i=0,\ldots,k$. 
Since we can check, integrating by parts, that
\[
 \int_0^1 x^iP_k(x){\rm d}x = 0\mbox{ for all } k=1,\ldots,q\mbox{ and } i=0,\ldots,k-1
\]
and
\[
 \int_0^1(P_k(x))^2{\rm d}x = 1,
\]
we get that the family of polynomials $(P_k)_{k=0,\ldots,q}$ is an orthonormal basis of $\mathcal{P}^q(\mathbb{R};\mathbb{R})$ (it is the basis issued from the Gram-Schmidt method applied to the basis $(x^i)_{i=0,\ldots,q}$). The relation $(\langle P_k,P_\ell\rangle)_{\substack{k=0,\ldots,q\\\ell=0,\ldots,q}} = P^t A P= {\rm Id}$ implies $A = (P^t)^{-1}P^{-1}$ and therefore $A^{-1} = P P^t$.
The expression
\[
  P_{ik} = (-1)^{k-i} \sqrt{2k+1}\mathcal{C}( k ,i)\mathcal{C}( k+i ,i)\mbox{ for all }k=0,\ldots,q \mbox{ and } i=0,\ldots,k
\]
concludes the proof of the lemma.
\end{proof}

\bibliographystyle{abbrv}

\end{document}